\documentclass[12pt]{amsart}
\usepackage{amssymb, amscd, amsmath, amsthm, comment}
\usepackage[foot]{amsaddr}
\usepackage{fullpage}
\usepackage{color}
\usepackage{pinlabel}
\usepackage{tikz}
\usepackage{hyperref}
\usepackage{todonotes}
\usepackage{colortbl}
\usepackage{multicol}
\usepackage{multirow}
\usepackage{hyperref}
\usepackage[normalem]{ulem}
\usetikzlibrary{knots}

\usetikzlibrary{arrows, shapes, decorations, decorations.markings, backgrounds, patterns, hobby, knots, calc, positioning}

\newtheorem{theorem}{Theorem}[section]
\newtheorem{lemma}[theorem]{Lemma}

\newtheorem{prop}[theorem]{Proposition}

\newtheorem{procedure}[theorem]{Procedure}

\theoremstyle{definition}
\newtheorem{definition}[theorem]{Definition}

\theoremstyle{remark}
\newtheorem{rem}[theorem]{Remark}

\newtheorem{example}[theorem]{Example}

\newcommand\bbm{\begin{bmatrix}}
\newcommand\ebm{\end{bmatrix}}

\newcommand{\Mod}[1]{\ (\mathrm{mod}\ #1)}
\newcommand\Gray[1]{\color{gray}#1\color{black}}

\makeatletter
\renewcommand{\author}[2][]{%
  \def\@tempa{#1}
  \ifx\@empty\authors
    \ifx\@tempa\@empty
      \gdef\shortauthors{#2}%
    \else
      \gdef\shortauthors{#1}%
    \fi
    \gdef\authors{\author{#2}}%
  \else
    \ifx\@tempa\@empty
      \g@addto@macro\shortauthors{\and#2}%
    \else
      \g@addto@macro\shortauthors{\and#1}%
    \fi
    \g@addto@macro\authors{\and\author{#2}}%
  \fi
}
\renewcommand{\address}[2][]{\g@addto@macro\authors{\address{#1}{#2}}}
\def\@setauthors{%
  \begin{center}%
    \footnotesize
    \vspace{20pt}
    \let\and\@empty
    \def\author##1{\advance\@tempcnta\@ne}%
    \def\address##1##2{\advance\@tempcntb\@ne}%
    \@tempcnta=\z@  \@tempcntb=\z@
    \authors
    \ifnum\@tempcnta>\@ne \ifnum\@tempcntb=\@ne
        \oneaddress
      \else
        \sepaddresses
      \fi
    \else
      \oneaddress
    \fi
  \end{center}%
}
\def\oneaddress{%
  \begingroup
  \let\author\@iden \let\address\@gobbletwo
  \renewcommand{\andify}{%
    \nxandlist{\unskip, }{\unskip{} and~}{\unskip, and~}}%
  \uppercasenonmath\authors
  \andify\authors
  \authors
  \endgroup
  \begingroup \let\and\relax \let\author\@gobble
  \def\address##1##2{\unskip\\[10pt] \itshape##2}%
  \authors
  \endgroup
}
\def\sepaddresses{%
  \begingroup
    \baselineskip10\p@\relax
    \def\address##1##2{ ({\itshape##2}\/)}
    \def\author##1{\def\temp{##1}\leavevmode\uppercasenonmath\temp\temp}%
    \nxandlist
      {,\\[\baselineskip]}
      {\\[\baselineskip] \textsc{\lowercase{and}}\\[\baselineskip]}
      {,\\[\baselineskip]\textsc{\lowercase{and}}\\[\baselineskip]}
      \authors 
    \authors
  \endgroup
}
\def\maketitle{\par
  \@topnum\z@
  \@setcopyright
  \thispagestyle{firstpage}%
  \uppercasenonmath\shorttitle
  \ifx\@empty\shortauthors \let\shortauthors\shorttitle
  \else
    \newcommand{\@xuppercasenonmath}[1]{\toks@\@emptytoks
      \@xp\@skipmath\@xp\@empty##1$$%
      \edef##1{\@nx\protect\@nx\@upprep\the\toks@}}%
    \@xuppercasenonmath\shortauthors
    \def\@@and{AND}
    \renewcommand{\andify}{%
      \nxandlist{\unskip, }{\unskip{ }\@@and{ }}{\unskip, \@@and{ }}}%
    \andify\shortauthors
  \fi
  \@maketitle@hook
  \begingroup
  \@maketitle
  \endgroup
  \c@footnote\z@
  \@cleartopmattertags
}
\def\@maketitle{%
  \normalfont\normalsize
  \let\@makefntext\noindent
  \@adminfootnotes
  \ifx\@empty\addresses\else \@footnotetext{\@setotheraddresses}\fi
  \global\topskip68\p@\relax
  \@settitle
  \ifx\@empty\authors \else \@setauthors \fi
  \ifx\@empty\@dedicatory
  \else
    \baselineskip26\p@
    \vtop{\centering{\footnotesize\itshape\@dedicatory\@@par}%
      \global\dimen@i\prevdepth}\prevdepth\dimen@i
  \fi
  \toks@\@xp{\shortauthors}\@temptokena\@xp{\shorttitle}%
  \edef\@tempa{\@nx\markboth{\the\toks@}{\the\@temptokena}}\@tempa
  \@setabstract
  \normalsize
  \if@titlepage
    \newpage
  \else
    \dimen@34\p@ \advance\dimen@-\baselineskip
    \vskip\dimen@\relax
  \fi
} 
\renewcommand{\thanks}[1]{%
  \ifx\@empty\thankses
    \gdef\thankses{\thanks{#1}}%
  \else
    \g@addto@macro\thankses{\endgraf\thanks{#1}}%
  \fi}
\def\@setthanks{\def\thanks##1{\noindent##1\@addpunct.}\thankses}
\renewcommand{\curraddr}[2][]{%
  \ifx\@empty\addresses
    \gdef\addresses{\curraddr{#1}{#2}}%
  \else
    \g@addto@macro\addresses{\endgraf\curraddr{#1}{#2}}%
  \fi}
\renewcommand{\email}[2][]{%
  \ifx\@empty\addresses
    \gdef\addresses{\email{#1}{#2}}%
  \else
    \g@addto@macro\addresses{\endgraf\email{#1}{#2}}%
  \fi}
\renewcommand{\urladdr}[2][]{%
  \ifx\@empty\addresses
    \gdef\addresses{\urladdr{#1}{#2}}%
  \else
    \g@addto@macro\addresses{\endgraf\urladdr{#1}{#2}}%
  \fi}
\def\@setotheraddresses{%
  \def\curraddr##1##2{\noindent
    \emph{Current address\@ifnotempty{##1}{ of ##1}}:\space
      ##2\@addpunct.}%
  \def\email##1##2{\noindent
    \emph{E-mail addresses\@ifnotempty{##1}{ of ##1}}:\space
      \texttt{##2}}%
  \def\urladdr##1##2{\noindent
    \emph{WWW address\@ifnotempty{##1}{ of ##1}}:\space
      \texttt{##2}}%
  \addresses
}
\let\enddoc@text\relax
\makeatother

\title{Average crosscap number of a 2-bridge knot}
\author{ Moshe Cohen, Thomas Kindred, Adam M. Lowrance, \\Patrick D. Shanahan, and Cornelia A. Van Cott }
\email{cohenm@newpaltz.edu, tkindred@smith.edu, adlowrance@vassar.edu,\\ pshanahan@lmu.edu, cvancott@usfca.edu}

\begin{document}

\begin{abstract} We determine a simple condition on a particular state graph of an alternating knot or link diagram that characterizes when the unoriented genus and crosscap number coincide, extending work of Adams and Kindred.  Building on this same work and using continued fraction expansions, we provide a new formula for the unoriented genus of a 2-bridge knot or link.    We use recursion to obtain exact formulas for the average unoriented genus $\overline{\Gamma}(c)$ and average crosscap number $\overline{\gamma}(c)$ of all 2-bridge knots with crossing number $c$, and in particular we show that $\lim_{c\to\infty} \left(\frac{c}{3}+\frac{1}{9} - \overline{\Gamma}(c)\right) = \lim_{c\to\infty} \left(\frac{c}{3}+\frac{1}{9} - \overline{\gamma}(c)\right) = 0$.
\end{abstract}
\maketitle

\section{Introduction}

A spanning surface $F$ for a knot (or link) $L\subset S^3$ is a connected, compact surface with no closed components embedded in $S^3$ whose boundary is $L$.  Its {\it complexity} is $\beta_1(F)=\text{rank } H_1(F)$. Since $F$ is connected, its complexity is the number of cuts required to reduce $F$ to a disk. Because $S^3$ is orientable, $F$ is nonorientable if and only if $F$ is one-sided. 

To begin we recall that the \textit{classical genus} $g(L)$ of an oriented nonsplit link $L$ in $S^3$ is
\[g(L)=\min\{g(F)~|~F\text{ is an oriented surface spanning }L\},\]
where $g(F)$ is the genus of the surface $F$. By considering nonorientable spanning surfaces, Clark \cite{Clark_1978} defined the crosscap number and unoriented genus of a link, though he refers to the unoriented genus as the maximal Euler characteristic. 
 \begin{definition}\label{D:CC}
 The \textit{unoriented genus} $\Gamma(L)$ and \textit{crosscap number} $\gamma(L)$ of a link $L\subset S^3$ are
  \begin{align*}
  \Gamma(L)&=\min\{\beta_1(F)~|~F\text{ is a spanning surface for }L\}~\text{and}\\
 \gamma(L)&=\min\{\beta_1(F)~|~F\text{ is a nonorientable spanning surface for }L\}.
 \end{align*}
\end{definition}

If $F$ is an orientable spanning surface of $L$, then $\beta_1(F) = 2g(F) + |L| -1$ where $|L|$ is the number of components of $L$. Given any orientable surface $F\subset S^3$ spanning $L$, we can always obtain a nonorientable spanning surface via $F\natural\includegraphics[width=.5cm]{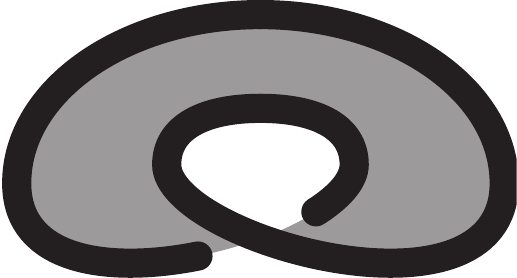}$,  where $\beta_1(F\natural \includegraphics[width=.5cm]{Figures/MobPos.pdf})=\beta_1(F)+1\geq2g(L)+|L|$. This implies the following inequality, due to Clark \cite{Clark_1978}:
\begin{equation*}
\label{ineq:CrosscapGenus}
\gamma(L) \leq 2g(L) + |L|.
\end{equation*}
The same reasoning implies that $\gamma(L)\in\{\Gamma(L),\Gamma(L)+1\}$ for any link $L$.

Haken \cite{Haken} and Schubert \cite{Schubert:genus} showed how to compute the classical genus of any knot using normal surface theory. Also, Ozsv\'ath and Szab\'o \cite{OS:genus} proved that the knot Floer homology of a knot determines its genus. If $L$ is alternating, then Murasugi \cite{mur58} and Crowell \cite{crowell} proved that applying Seifert's algorithm to an alternating diagram of $L$ results in a minimum genus Seifert surface (see also Gabai \cite{gab86a} and Kindred \cite{Kindred2024}).

The crosscap number is known for several infinite families of knots and links. Teragaito \cite{Teragaito_2004} computed the crosscap number of torus knots. Ichihara and Mizushima \cite{IM:2005} found the crosscap number of most pretzel knots. Several other groups of authors found various ways to compute the crosscap number of 2-bridge knots. We discuss these methods in more detail below. More generally, Jaco, Rubinstein, Spreer, and Tillman \cite{JRST:CC} improved an algorithm of Burton and Ozlen \cite{BO:crosscap} that uses normal surface theory to compute the crosscap number of an arbitrary knot. In a different direction, Kalfagianni and Lee \cite{Kalfagianni_Lee:2016} gave two-sided linear bounds on the crosscap number of alternating links in terms of the second and penultimate coefficients of the Jones polynomial. McConkey \cite{McConkey} found similar bounds for links obtained as Conway sums of strongly alternating tangles.

There are several related procedures for finding the unoriented genus and crosscap number of an alternating link, all using (Kauffman) state surfaces in some fashion. A \textit{Kauffman state} of a link diagram is the collection of planar curves obtained by resolving each crossing  $\tikz[baseline=.6ex, scale = .4]{
\draw (0,0) -- (1,1);
\draw (1,0) -- (.7,.3);
\draw (.3,.7) -- (0,1);
}$ as an $A$-resolution  $\tikz[baseline=.6ex, scale = .4]{
\draw[rounded corners = 1mm] (0,0) -- (.45,.5) -- (0,1);
\draw[rounded corners = 1mm] (1,0) -- (.55,.5) -- (1,1);
}$ or a $B$-resolution $\tikz[baseline=.6ex, scale = .4]{
\draw[rounded corners = 1mm] (0,0) -- (.5,.45) -- (1,0);
\draw[rounded corners = 1mm] (0,1) -- (.5,.55) -- (1,1);
}$.  The \textit{state surface} associated with a Kauffman state $x$ is the surface constructed by filling in components of $x$ with disks and replacing crossings with twisted bands. See Figure~\ref{Fi:Plumbing} for an example of a Kauffman state and its associated state surface.  Adams and Kindred \cite{Adams-Kindred:2013} proved if $D$ is an alternating diagram of $L$ with $c>0$ crossings, then there are state surfaces $F$ and $F'$ (not necessarily distinct) such that $F$ is non-orientable, $\beta_1(F) = \gamma(L)$, and $\beta_1(F')=\Gamma(L)$. Therefore one can compute the crosscap number and unoriented genus of a nonsplit alternating link by examining the $2^c$ state surfaces of its alternating diagram.

Building on work of various combinations of Adams, Kindred, Ito, Takimura, and Yamada  \cite{Adams-Kindred:2013,Kindred:2019,Ito-Takimura:2020,Yamada_Ito_2021}, Kindred gave an algorithm that computes the crosscap number of an alternating knot and implemented it through 13 crossings  \cite{Kindred:2019}, while Yamada and Ito gave another such algorithm which runs in $O(c^3)$ \cite{Yamada_Ito_2023}.  Ito and Takimura \cite{Ito_Takimura_2018, Ito_Takimura_2022} classified alternating diagrams whose crosscap number is two or three.

Given a $c$-crossing alternating diagram $D$ of a link $L$, Adams and Kindred also described a procedure that produces a state surface $F$ of $D$ which realizes $\Gamma(L)$, but possibly not $\gamma(L)$, and that does not require examining all $2^c$ Kauffman states (see Procedure~\ref{Proc:AK Alt} and \cite[Theorem 3.3]{Adams-Kindred:2013}).
Our first theorem determines the crosscap number $\gamma(L)$ from this surface $F$ obtained by this procedure of Adams and Kindred.

The \textit{state graph} $G_x$ of the Kauffman state $x$ of the link diagram $D$ is the graph whose vertices and edges correspond to the components of the Kauffman state $x$ and the crossings of $D$, respectively, where an edge is incident to a vertex in $G_x$ if the corresponding crossing is incident to the corresponding component of the Kauffman state. We prove:

\begin{theorem}\label{T:CC}
Let $D$ be an alternating diagram of a nontrivial link $L\subset S^3$ with no Hopf link connect summand, and let $F_x$ be a state surface from $D$ with $\beta_1(F_x)=\Gamma(L)$. Let $G_x$ be its state graph. 
Then $\gamma(L)=\Gamma(L)+1$ if and only if every cycle in $G_x$ has an even length of at least 4.
\end{theorem}
Because adding a Hopf link connect summand increases the crosscap number of a link by one, Theorem~\ref{T:CC} can be used to find the crosscap number of any alternating link.

Every nonzero rational number $\frac{p}{q}$ with $-1<\frac{p}{q} < 1$ can be represented as a continued fraction
\begin{equation}
\label{eq:contfrac}
    \frac{p}{q}=\cfrac{1}{a_1+\cfrac{1}{a_2+\cfrac{1}{\ddots +\cfrac{1}{a_k}}}},
\end{equation}

or in more compact notation $\frac{p}{q} = [a_1,\dots,a_k]$, where each $a_i$ is a nonzero integer. The $2$-bridge link $L(p/q)$ is defined to be the link with a $4$-plat diagram as in Figure~\ref{Fi:fourplat} where a box labeled $a_i$ indicates a twist region with $|a_i|$ crossings, and the choice of right- or left-handed crossing is determined by the sign of $a_i$, as in the table. If $0< p/q<1$, then $p/q$ can be represented by a continued fraction of the form $[b_1,-b_2,b_3,-b_4,\dots, (-1)^{k-1} b_k]$ where $b_i\geq 2$. We call this representation the \textit{positive subtractive form} of $p/q$ and abbreviate it 
$p/q = [b_1,\dots,b_k]_-$. 

Schubert \cite{Schubert_1954} proved that the 2-bridge links $L(\frac{p}{q})$ and $L(\frac{p-q}{q})$ are equivalent, and $\frac{p}{q}$ or $\frac{p-q}{q}$ can be represented as a continued fraction $[e_1,-e_2,e_3,-e_4,\dots,(-1)^m e_m]_+$ where each $e_i$ is nonzero and even.  We call this representation an \textit{even subtractive form} of $p/q$ and abbreviate it 
$[e_1,\hdots,e_m]_-$.
The link $L(p/q)$ is a knot precisely when $q$ is odd. In that case, 
the even subtractive form of $p/q$ is unique. See Section~\ref{sec:background} for more discussion of continued fractions representing 2-bridge knots and links.

Hatcher and Thurston \cite{Hatcher-Thurston:1985} found all incompressible spanning surfaces of $2$-bridge links. Hirasawa and Teragaito \cite{Hirasawa_Teragaito:2006} gave an algorithm to compute the crosscap number of the 2-bridge knot $K(p/q)$ in terms of the minimum length of a continued fraction expansion of $p/q$ subject to certain constraints (throughout the paper we use $K$ when referring to knots alone and use $L$ when referring to knots or links). Hoste, Shanahan, and Van Cott \cite{Hoste-Shanahan-VanCott:2021} computed the crosscap number of the 2-bridge knot $K(p/q)$ in terms of the depth of the rational number $p/q$ in the Farey graph. Since every 2-bridge link is alternating, we can use Theorem~\ref{T:CC} to find a new formula for the unoriented genus and crosscap number of a $2$-bridge link.

\begin{theorem}
    \label{thm:cc_formula}
    Let $L(p/q)$ be a 2-bridge link where $0<p/q<1$, the fraction $p/q$ has positive subtractive continued fraction $[b_1,\dots,b_k]_-$, and $p/q$ has even subtractive continued fraction $[e_1,
    \dots,e_\ell]_-$.  Define
    \begin{itemize}
        \item $w=\#\{i~:b_i=2\text{ and either }i=1\text{ or }b_{i-1}\neq2\}+\#\{i:~b_i\geq3\}$ and 
        \item $z=\#\{i:~\text{for some }j\geq 0,~b_i=3=b_{i+1}=\cdots=b_{i+j}\text{ and }b_{i-1}=2=b_{i+j+1}\}$. 
    \end{itemize}
    Then the unoriented genus of $L(p/q)$ is $\Gamma(L(p/q)) = w-z$, and the crosscap number of $L(p/q)$ is
    \[\gamma(L(p/q)) = \begin{cases} 
    \Gamma(L(p/q)) & \text{if $|e_i|=2$ for some $i$,}\\
    \Gamma(L(p/q))+1 & \text{if $|e_i|\geq 4$ for all $i$.}
    \end{cases}\]
\end{theorem}
In Theorem~\ref{thm:cc_formula}, the first quantity in $w$ is the number of maximal substrings of $2$'s in $[b_1,\dots,b_k]_-$ and the second quantity is the number of entries in $[b_1,\dots,b_k]_-$ that are at least 3. The quantity $z$ is the number of substrings of length at least three in $[b_1,\dots,b_k]_-$ that start and end with $2$ and where every other entry is $3$. Also as discussed in Section~\ref{sec:background}, every 2-bridge knot corresponds to a rational number $p/q$ satisfying the assumptions of Theorem~\ref{thm:cc_formula}.

A growing body of work investigates the probabilistic properties of different invariants of the set of $2$-bridge knots with fixed crossing number. Motivated by computational data of Dunfield et.~al \cite{Dunfield:2014}, Cohen \cite{Coh:lower} found a lower bound on the average genus $\overline{g}_c$ of 2-bridge knots with a fixed crossing number $c$. Suzuki and Tran \cite{SuzukiTran} and independently Cohen and Lowrance \cite{CohLow} computed the average genus $\overline{g}_c$ exactly (see also Ray and Diao \cite{RayDiao}). Cohen et.~al \cite{CohLowURSI} found the variance, median, and mode of the genera of $2$-bridge knots with fixed crossing number and proved that this probability distribution is asymptotically normal. Suzuki and Tran \cite{SuzukiTran:Braid} computed the average braid index of $2$-bridge knots with a fixed crossing number. Clark, Frank, and Lowrance \cite{CFL:Braid} independently computed the average, variance, and mode of the braid indices of $2$-bridge knots with a fixed crossing number and gave a formula for the number of $2$-bridge knots with a fixed braid index and crossing number. Baader, Kjuchukova, Lewark, Misev, and Ray \cite{BKLMR} proved that the average $4$-genus of $2$-bridge knots is sublinear with respect to a parameter that approximates the crossing number. Cohen, Lowrance, Madras, and Raanes \cite{CLMR} computed the average absolute value of the signature of $2$-bridge knots with a fixed crossing number and found explicit sublinear lower and upper bounds on the $4$-genus of $2$-bridge knots with a fixed crossing number.

Let $\mathcal{K}_c$ be the set of $2$-bridge knots with crossing number $c$ where if a knot $K$ is chiral of crossing number $c$, both $K$ and its mirror $\overline{K}$ are elements of $\mathcal{K}_c$. For example, $\mathcal{K}_3$ contains both the left- and right-handed trefoils, and $\mathcal{K}_4$ contains only the figure-eight knot. Ernst and Sumners \cite{Ernst-Sumners:1987} first computed $|\mathcal{K}_c|$, the number of 2-bridge knots with crossing number $c$ (see Theorem~\ref{T:Ernst-Sumners}). Define the \textit{average unoriented genus} $\overline{\Gamma}(c)$ and the \textit{average crosscap number} $\overline{\gamma}(c)$ of the set of 2-bridge knots with crossing number $c$ by
\[\overline{\Gamma}(c) = \frac{\sum_{K\in\mathcal{K}_c} \Gamma(K)}{|\mathcal{K}_c|}~\text{and}~\overline{\gamma}(c) = \frac{\sum_{K\in\mathcal{K}_c} \gamma(K)}{|\mathcal{K}_c|},\]
respectively. In our next theorem, we compute the average unoriented genus and average crosscap number of $2$-bridge knots with a fixed crossing number.
\begin{theorem}
\label{T:ccave}
The average unoriented genus $\overline{\Gamma}(c)$ and the average crosscap number $\overline{\gamma}(c)$ of the set of $2$-bridge knots with crossing number $c$ satisfy
\[\overline{\Gamma}(c)=\frac{c}{3}+\frac{1}{9} +\varepsilon_1(c)~\text{and}~\overline{\gamma}(c)= \overline \Gamma(c) + \varepsilon_2(c),\]
where $\varepsilon_1(c)$ and $\varepsilon_2(c)$ approach $0$ as $c$ approaches infinity.
\end{theorem}
Theorems~\ref{T:Ep1} and \ref{T:Ep2} imply Theorem \ref{T:ccave} and give exact formulas for $\varepsilon_1(c)$ and $\varepsilon_2(c)$.

We prove Theorem~\ref{T:ccave} by taking advantage of recursive structures in certain sets of 2-bridge knot diagrams. Roughly, we define the set $K(c)$ to be the set of $2$-bridge knot diagrams coming from positive subtractive continued fractions that have crossing number $c$. We find recursive formulas for the totals $W(c)$ and $Z(c)$ of $w$ and $z$ respectively over all diagrams in $K(c)$. These recursive formulas lead to closed formulas which allow us to compute the asymptotics of average unoriented genus using Theorem~\ref{thm:cc_formula}. We similarly use a set of 2-bridge knot diagrams coming from even subtractive continued fractions with crossing number $c$ to count the number of $2$-bridge knots with $c$ crossings where $\gamma(K) = \Gamma(K)+1$, which yields the asymptotic formula for the average crosscap number. {The authors expect that Theorem~\ref{T:ccave} can be extended to 2-bridge links using similar methods.}

The paper is organized as follows. In Section~\ref{sec:background}, we describe background information about 2-bridge knots and state surfaces. In Section~\ref{sec:algorithm}, we prove Theorems~\ref{T:CC} and \ref{thm:cc_formula}, giving a new method to compute the crosscap number of an alternating link and a new formula for the crosscap number of a $2$-bridge knot, respectively. In Section~\ref{sec:unoriented_genus}, we find a formula for the average unoriented genus $\overline{\Gamma}(c)$ of $2$-bridge knots with crossing number $c$. Finally, in Section~\ref{sec:epsilon}, we find the formula for the $\varepsilon_2(c)$ term appearing in Theorem~\ref{T:ccave}.

\textbf{Acknowledgements. } The authors would like to thank the partners of the Joint Mathematics Meetings, where this project began in 2024.  

\section{Background}
\label{sec:background}

{In this section, we review various facts about 2-bridge knots and spanning surfaces, with a focus on spanning surfaces of 2-bridge knots as described by Hatcher and Thurston.}

\subsection{Continued fractions for 2-bridge knots and links}\label{S:CF}

\begin{figure}[h]
\includegraphics[width = 4in]{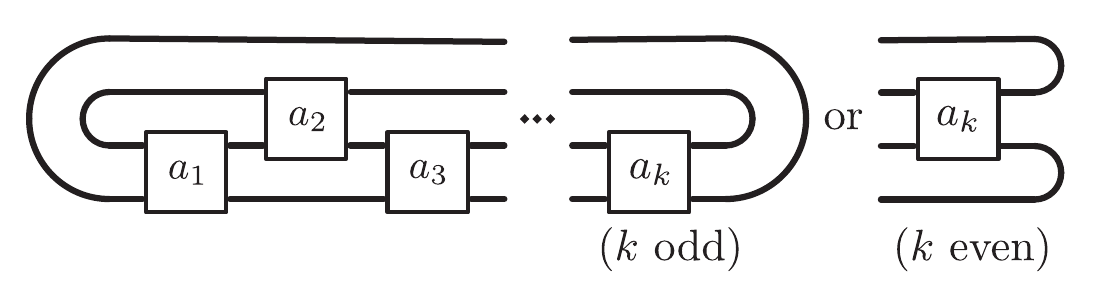}\\
\begin{tabular}{|l|c|c|} \hline
 & $a_i > 0$ &  $a_i < 0$ \\ \hline
$i$ odd &right-hand $\tikz[baseline=.6ex, scale = .4, thick]{
\draw (0,1) -- (1,0);
\draw (0,0) -- (.3,.3);
\draw (.7,.7) -- (1,1);
}$  & left-hand $\tikz[baseline=.6ex, scale = .4, thick]{
\draw (0,0) -- (1,1);
\draw (1,0) -- (.7,.3);
\draw (.3,.7) -- (0,1);
}$ \\ \hline
$i$ even & left-hand $\tikz[baseline=.6ex, scale = .4, thick]{
\draw (0,0) -- (1,1);
\draw (1,0) -- (.7,.3);
\draw (.3,.7) -- (0,1);
}$ & right-hand $\tikz[baseline=.6ex, scale = .4, thick]{
\draw (0,1) -- (1,0);
\draw (0,0) -- (.3,.3);
\draw (.7,.7) -- (1,1);
}$ \\ \hline
\end{tabular}

\caption{A 4-plat diagram of a 2-bridge link.} 
\label{Fi:fourplat}
\end{figure}
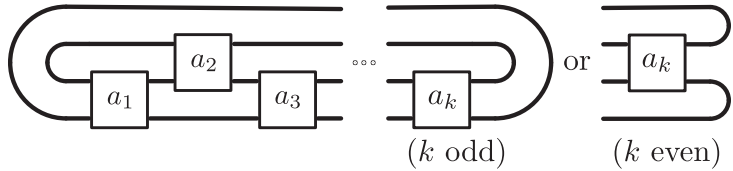

In general, a  {2-bridge link diagram} $D$ is the plat closure of a 4-braid. Let $\sigma_1,\sigma_2,\sigma_3$ denote the usual generators of the group of 4-braids. This diagram $D$ admits a sequence of flype moves which removes all appearances of $\sigma_3$. If the resulting diagram is reduced (that is, the diagram admits no nugatory crossings), then we can arrange for the braid to begin with a nonzero power of $\sigma_1$.  Thus, every 2-bridge link has a diagram as in Figure~\ref{Fi:fourplat}. 

We denote the link with the diagram in Figure~\ref{Fi:fourplat} by $L[a_1,\hdots,a_k]$. This diagram is alternating if and only if all $a_i$ have the same sign. Kauffman \cite{Kauffman1987}, Murasugi \cite{Murasugi1987}, and Thistlethwaite \cite{Thistlethwaite1987} showed that a reduced alternating link diagram minimizes its crossing number, and since the 4-plat diagram of $L[a_1,\dots,a_k]$ is reduced and alternating when each $a_i$ has the same sign, the crossing number of the link $L[a_1,\dots,a_k]$ is $c=|a_1 +a_2 +\dots +a_k|$.

In general, a vector of integers $( a_1, 
\dots, a_k )$ determines a fraction $p/q$ via the \emph{additive} continued fraction $\frac{p}{q}=[a_1, 
\dots, a_k]$ as in Equation \eqref{eq:contfrac}.
Then we write $L[a_1,\ldots,a_k]$ to denote the link $L(p/q)$. Schubert \cite{Schubert_1956} showed that when $p/q$ and $p'/q'$ are reduced fractions with $q,q'>0$, the links $L(p/q)$ and $L(p'/q')$ are ambiently isotopic if and only if $q = q'$ and $p \equiv p' \Mod{q}$ or $pp' \equiv 1 \Mod{q}$. Hence, excluding the case of the trivial 2-component link $L(0)$, we may always assume when writing $L(p/q)$ that $0<p<q$. 
Also, it turns out that $L(p/q)$ is a knot if and only if $q$ odd; otherwise it is a 2-component link. 
See \cite{BS:2010,BZ:2003} for a proof. 

A reduced fraction $p/q$ with $0< p < q$ can be given by many different additive continued fractions, but
the division algorithm and the fact that \[
[a_1,\dots, a_{k-1},a_k,1] = [a_1, \dots, a_{k-1},a_k+1]
\]imply  that there is a unique odd-length representation $p/q = [a_1,\dots, a_{2k+1}]$ in which every $a_i >0$. Call this the \textit{positive additive} continued fraction representation of $p/q$, and denote it by $[a_1,\dots, a_{2k+1}]_+$. From Figure~\ref{Fi:fourplat}, it is clear that reflecting the diagram for $L[a_1,\dots, a_{2k+1}]_+$ in the plane of the diagram gives $L[-a_1,\dots, -a_k]$. Moreover, by rotating the diagram for $L[a_1,\dots, a_{2k+1}]_+$ by 180$^\circ$ and then applying flypes to change $\sigma_3$ crossings into $\sigma_1$ crossings, we see that $L[a_1,\dots, a_{2k+1}]_+$ is equivalent to $L[a_{2k+1},\dots, a_1]_+$. This observation can also be derived algebraically by verifying that if $[a_1,\dots, a_{2k+1}]_+ = p/q$ and $ [a_{2k+1},\dots, a_1]_+ = p'/q'$, then $q = q'$ and $pp' \equiv 1 \Mod{q}$ (see \cite{BZ:2003}).

There are also \textit{subtractive} ways to represent ${p/q}$ as a continued fraction:

\[\frac{p}{q}=\cfrac{1}{b_1-\cfrac{1}{b_2-\cfrac{1}{\ddots -\cfrac{1}{b_\ell}}}}.
\]

We abbreviate this using the notation $p/q=[b_1,\hdots,b_\ell]_-$. 
In particular, we can convert a positive additive continued fraction $[a_1,\dots, a_{2k+1}]_+=p/q$ to a subtractive continued fraction $[b_1,\hdots,b_\ell]_-=p/q$ where each $b_i\geq 2$, which we call the \textit{positive subtractive} continued fraction representation of $p/q$.  To obtain $[b_1,\hdots,b_\ell]_-$ from $[a_1,\dots, a_{2k+1}]_+$, replace $a_1$ and $a_{2k+1}$ with $a_1+1$ and $a_{2k+1}+1$, replace each remaining odd-index coefficient $a_{2i+1}$ with $a_{2i+1}+2$, and replace each
even-index coefficient $a_{2i}$ with a string of ($a_{2i}-1$) 2's (possibly none).
In simpler notation we have
\[
[a_1,
\dots, a_{2k+1}]_+ = [a_1+1, 2^{[a_2 -1]},a_3+2, 2^{[a_4-1]}, \dots, a_{2k-1}+2, 2^{[a_{2k}-1]},a_{2k+1}+1]_-
\]
where $2^{[n]}$ represents a string of $n$ 2's (including an empty string if $n=0$). The coefficients of this subtractive form can be read from the 4-plat diagram $D$ for the additive positive continued fraction by taking the checkerboard surface of $D$ in which the unbounded region is unshaded and counting the number of crossings around each bounded unshaded region. Note here that $D$ is alternating. See Figure~\ref{subtractiveform}.

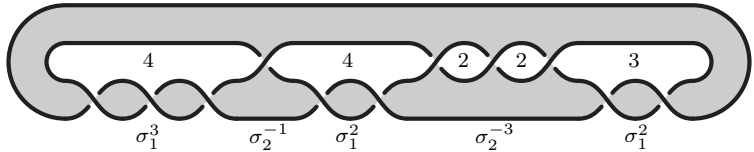
\begin{figure}[h]
\labellist\tiny
\pinlabel {4} at 270 115
\pinlabel {4} at 650 115
\pinlabel {2} at 870 115
\pinlabel {2} at 980 115
\pinlabel {3} at 1195 115
\pinlabel {$\sigma_1^3$} at 270 -30
\pinlabel {$\sigma_2^{-1}$} at 500 -30
\pinlabel {$\sigma_1^2$} at 650 -30
\pinlabel {$\sigma_2^{-3}$} at 930 -30
\pinlabel {$\sigma_1^{2}$} at 1200 -30
\endlabellist
\includegraphics[width=.6\textwidth]{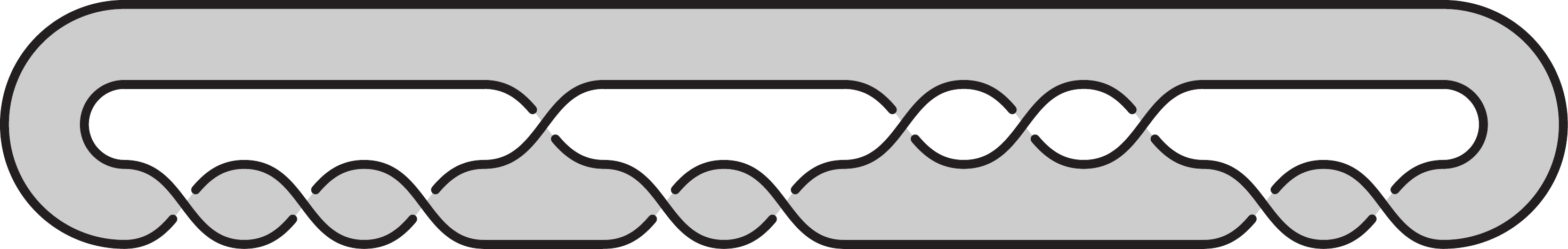}
\caption{The additive and subtractive forms $23/85=[3,1,2,3,2]_+ = [4,4,2,2,3]_-$.} 
\label{subtractiveform}
\end{figure}

In Figure~\ref{subtractiveform}, we see that, whereas a positive additive continued fraction representation describes a 2-bridge knot and its diagram as a plat closure of a braid, a subtractive continued fraction representation describes a 2-bridge knot and its diagram via a spanning surface $F$.  In subsection~\ref{S:HT}, we explain how the latter description applies to any 2-bridge  link $L(p/q)$ and any subtractive continued representation $[b_1,\hdots,b_k]_-=p/q$ in which each $|b_i|\geq 2$. This follows work of Hatcher and Thurston \cite{Hatcher-Thurston:1985}, although the diagrammatic aspect of our description appears to be new.

\subsection{Spanning surfaces}

Recall that a spanning surface $F$ for a  link $L\subset S^3$ is a connected, compact surface embedded in $S^3$ whose boundary is $L$ and that each Kauffman state $x$ of a diagram $D$ of $L$ has an associated state surface $F_x$ obtained by gluing twisted bands that correspond to crossings to disks the correspond to the state circles of $x$. The state $x$ and its surface $F_x$ are called {\it adequate} if each crossing arc joins distinct state circles. Ozawa proved that every adequate state surface from an alternating diagram is incompressible and $\partial$-incompressible \cite{Ozawa:2011}.  

From the Kauffman state $x$, one can also construct a {\it state graph} $G_x$, whose vertices correspond to the state circles of $x$ and whose edges correspond to the crossing bands in $F_x$.  The graph $G_x$ embeds in $F_x$ as its spine. See Figure~\ref{Fi:Plumbing}.

If $D$ has $c$ crossings and $x$ has $|x|$ state circles, then $\chi(F_x)=|x|-c$, and $\beta_1(F_x)=1+c-|x|.$ Thus, choosing a minimum-complexity state surface from $D$ amounts to choosing a state with the maximum number of state circles.

Every link diagram $D$ has two special Kauffman states, known as the \textit{checkerboard states}, constructed as follows. First choose a black and white checkerboard coloring of $S^2\setminus D$ where two complementary regions separated by an edge of $D$ are assigned different colors. One checkerboard state is obtained by choosing resolutions at every crossing so that the state circles form the boundaries of the black regions, and the other checkerboard state is obtained similarly with state circles bounding the white regions. If $B$ and $W$ are the state surfaces of the black and white checkerboard states respectively, then $B$ projects to the black regions, $W$ projects to the white regions, and the interiors of $B$ and $W$ intersect in vertical arcs, one at each crossing. The state graphs $G_B$ and $G_W$ of the black and white states are known as the \textit{Tait} or \textit{checkerboard graphs} of $D$. They are dual planar graphs.

\begin{figure}
\begin{center}
\includegraphics[width=\textwidth]{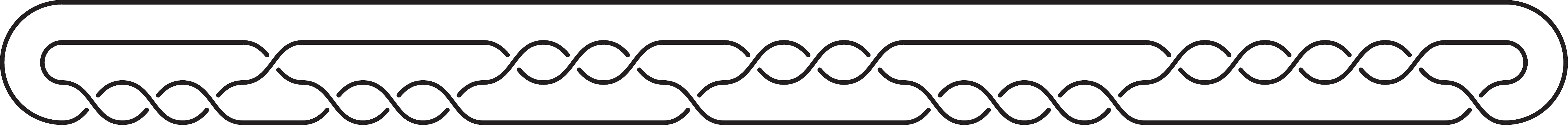}\\
\labellist \tiny
\pinlabel{$A$} at 926 90
\pinlabel{$A$} at 1033 90
\pinlabel{$A$} at 1140 90
\pinlabel{$A$} at 1356 90
\pinlabel{$A$} at 1463 90
\pinlabel{$A$} at 1570 90
\pinlabel{$B$} at 468 108
\pinlabel{$B$} at 383 61
\pinlabel{$B$} at 276 61
\pinlabel{$B$} at 169 61
\pinlabel{$B$} at 600 61
\pinlabel{$B$} at 708 61
\pinlabel{$B$} at 815 61
\pinlabel{$A$} at 1275 45
\pinlabel{$B$} at 1684 61
\pinlabel{$B$} at 1792 61
\pinlabel{$B$} at 1900 61
\pinlabel{$B$} at 2008 61
\pinlabel{$A$} at 2114 90
\pinlabel{$A$} at 2225 90
\pinlabel{$A$} at 2332 90
\pinlabel{$A$} at 2440 90
\pinlabel{$A$} at 2548 90
\pinlabel{$A$} at 2676 45
\endlabellist
\includegraphics[width=\textwidth]{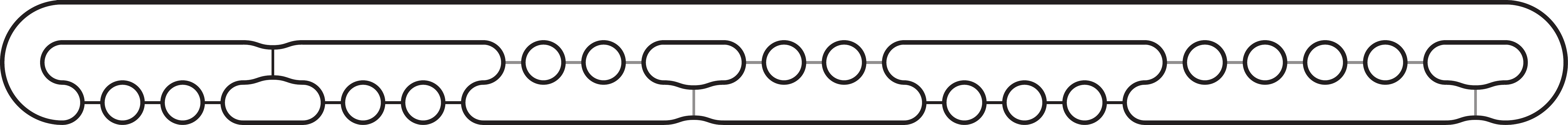}\\
\includegraphics[width=\textwidth]{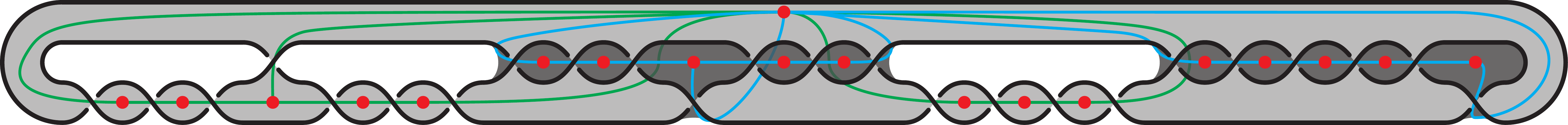}
\caption{A diagram $D$, state $x$, state surface $F_x$, and state graph $G_x$. }
\label{Fi:Plumbing}
\end{center}
\end{figure}
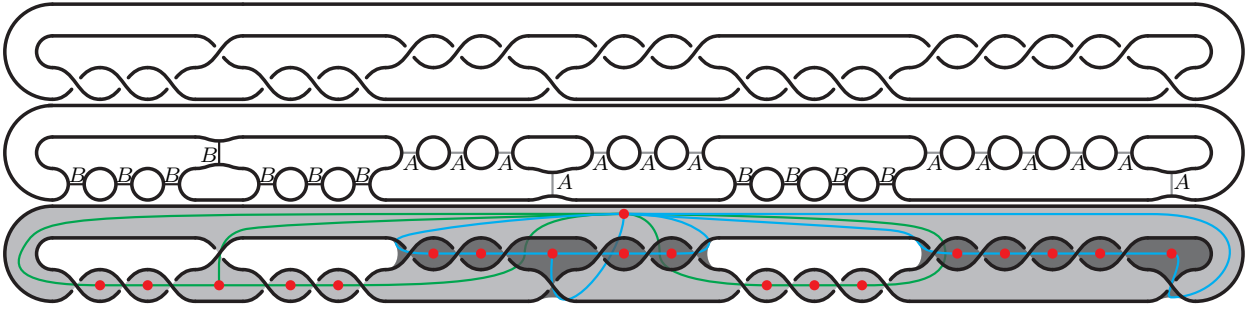

\subsection{Hatcher-Thurston surfaces}\label{S:HT}

Hatcher and Thurston described how every subtractive continued fraction representation $[b_1,\hdots,b_k]_-$ of $p/q$ yields a family of spanning surfaces for $L(p/q)$ \cite{Hatcher-Thurston:1985}.  Each surface in this family is a plumbing, or Murasugi sum, of surfaces $F_1,\hdots,F_k$, where each $F_i$ is an unknotted annulus or M\"obius band whose core circle has a framing of $b_i/2$; the plumbing is performed so that $F_i$ is glued along squares to $F_{i+1}$ and $F_{i-1}$ for each $i=2,\hdots,n-1$, as in Figure \ref{Fi:2BridgePlumbings}. The resulting surface $F=F_1*\cdots*F_k$ is not necessarily well-defined up to isotopy, because there are two ways to plumb each $F_{i+1}$ onto $F_1*\cdots *F_i$. (The ways to plumb $F_{i+1}$ onto $F_1*\cdots *F_i$ are equivalent if and only if $F_i$ or $F_{i+1}$ is a Hopf band.) The surface $F$ is incompressible and boundary-incompressible if and only if each $|b_i|\geq 2$. We assume this is the case for the rest of \textsection\ref{S:HT}.

In fact, given coprime $p,q$ with $0<p<q$, there is a reduced alternating diagram $D$ of $L(p/q)$ such that every subtractive continued fraction representation $\mathbf b=[b_1,\hdots,b_k]_-$ of $p/q$ or $(p-q)/q$ with each $|b_i|\geq 2$ corresponds to an adequate state surface $F_{\mathbf b}$ of $D$. Given such $\mathbf b$ we construct $F_{\mathbf b}=F_1*\cdots*F_k$ and $D$ as follows.

If $b_i>0$, we draw $F_i$ as a checkerboard surface as shown in Figure~\ref{Fi:PosBand}, whereas if $b_i<0$, we draw $F_i$ as shown in Figure~\ref{Fi:NegBand}.
If $b_{i-1}>0$, we draw $F_i$ as shown second or third in Figure~\ref{Fi:PosBand} or \ref{Fi:NegBand}; otherwise, $i=1$ or $b_{i-1}<0$, and we draw $F_i$ as shown first or last (with ``nested C's'' on the left).  If $b_{i+1}>0$, we draw $F_i$ as shown first or second in Figure~\ref{Fi:PosBand} or \ref{Fi:NegBand}; otherwise, $i=n$ or $b_{i+1}<0$, and we draw $F_i$ as shown third or fourth (with nested C's on the right). Note that in all cases we draw $F$ with nested C's on the left and right.

\begin{figure}
\begin{center}
 \labellist
\tiny\hair 4pt
\pinlabel $5$ [c] at 330 110
\pinlabel $5$ [c] at 1060 110
\pinlabel $5$ [c] at 1800 110
\pinlabel $5$ [c] at 2580 110
\endlabellist
\includegraphics[width=\textwidth]{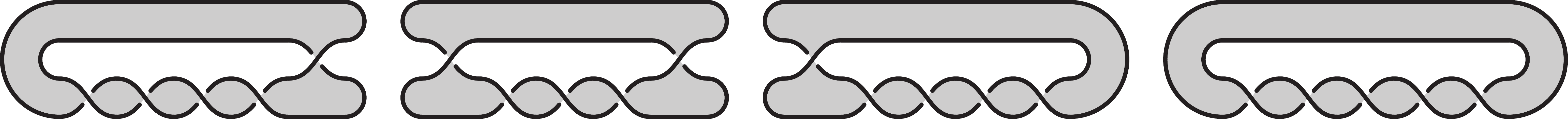}
\caption{The band $F_{i}$ when $b_i>0$ (shown here when $b_i=5$)} 
\label{Fi:PosBand}
\end{center}
\end{figure}

\begin{figure}
\begin{center}
 \labellist
\tiny\hair 4pt
\pinlabel $-5$ [c] at 440 45
\pinlabel $-5$ [c] at 1050 45
\pinlabel $-5$ [c] at 1700 45
\pinlabel $-5$ [c] at 2600 45
\endlabellist
\includegraphics[width=\textwidth]{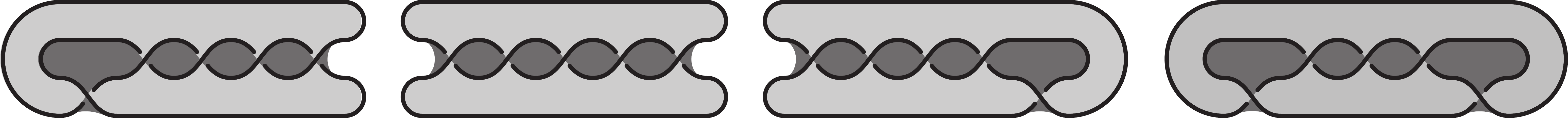}
\caption{The band $F_{i}$ when $b_i<0$ (shown here when $b_i=-5$)} 
\label{Fi:NegBand}
\end{center}
\end{figure}

Each $F_{i+1}$ is plumbed onto $F_i$ in $F_1*\cdots*F_i$ in one of the four ways shown in Figure~\ref{Fi:2BridgePlumbings}, depending on the signs of $b_i$ and $b_{i+1}$. Figures~\ref{Fi:EvenCB} and \ref{Fi:PosCB} both show full examples. It is interesting to note that, under this construction, the only crossings that are smoothed horizontally in a Kauffman state are those where two bands of the same sign are plumbed together; the rest are smoothed vertically.

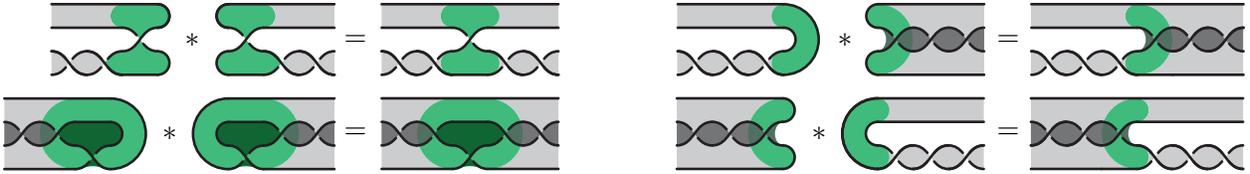
\begin{figure} 
\begin{center}
 \labellist
\small\hair 4pt
\pinlabel $*$ [c] at 510 115
\pinlabel $=$ [c] at 1080 115
\pinlabel $*$ [c] at 580 400
\pinlabel $=$ [c] at 1080 400 
\pinlabel $*$ [c] at 2500 115
\pinlabel $=$ [c] at 3080 115
\pinlabel $*$ [c] at 2580 400
\pinlabel $=$ [c] at 3080 400  
\endlabellist
\includegraphics[width=\textwidth]{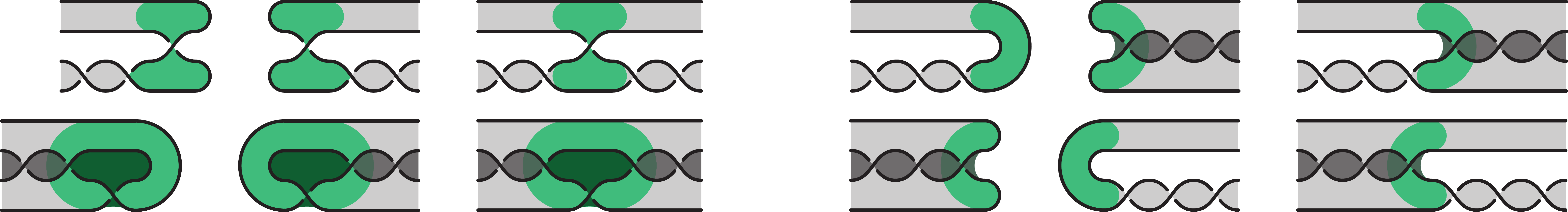}
\caption{Plumbing the band $F_{i+1}$ (center in each row) onto $F_1*\cdots*F_i$ (left)}
\label{Fi:2BridgePlumbings}
\end{center}
\end{figure}

\begin{figure}
\begin{center}
 \labellist
\tiny\hair 4pt
\pinlabel $4$ [c] at 270 360
\pinlabel $4$ [c] at 270 650
\pinlabel $4$ [c] at 700 110
\pinlabel $4$ [c] at 700 650
\pinlabel $-4$ [c] at 1020 300
\pinlabel $-4$ [c] at 1020 580
\pinlabel $-4$ [c] at 1460 40
\pinlabel $-4$ [c] at 1460 580
\pinlabel $4$ [c] at 1850 360
\pinlabel $4$ [c] at 1850 650
\pinlabel $-6$ [c] at 2400 40
\pinlabel $-6$ [c] at 2400 580 \endlabellist
\includegraphics[width=\textwidth]{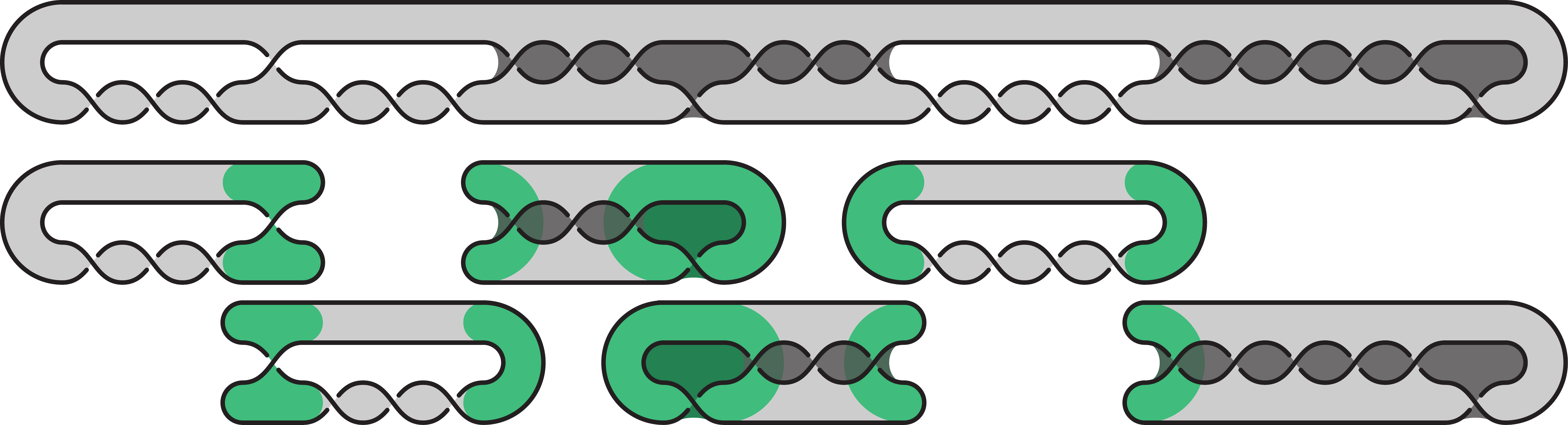}
\caption{The Hatcher--Thurston surface associated to  $[4,4, -4,-4,4,-6]_-$}
\label{Fi:EvenCB}
\end{center}
\end{figure}

\begin{figure}
\begin{center}
 \labellist
\tiny\hair 4pt
\pinlabel $4$ [c] at 385 652
\pinlabel $5$ [c] at 820 652
\pinlabel $2$ [c] at 1090 652
\pinlabel $2$ [c] at 1200 652
\pinlabel $3$ [c] at 1360 652
\pinlabel $2$ [c] at 1528 652
\pinlabel $2$ [c] at 1632 652
\pinlabel $6$ [c] at 1950 652
\pinlabel $2$ [c] at 2284 652
\pinlabel $2$ [c] at 2392 652
\pinlabel $2$ [c] at 2500 652
\pinlabel $2$ [c] at 2608 652
\pinlabel $2$ [c] at 2716 652
\pinlabel $4$ [c] at 275 362
\pinlabel $5$ [c] at 415 112
\pinlabel $2$ [c] at 745 362
\pinlabel $2$ [c] at 882 112
\pinlabel $3$ [c] at 1115 362
\pinlabel $2$ [c] at 1318 112
\pinlabel $2$ [c] at 1472 362
\pinlabel $6$ [c] at 1845 112
\pinlabel $2$ [c] at 2200 362
\pinlabel $2$ [c] at 2360 112
\pinlabel $2$ [c] at 2512 362
\pinlabel $2$ [c] at 2665 112
\pinlabel $2$ [c] at 2872 362
\endlabellist
\includegraphics[width=\textwidth]{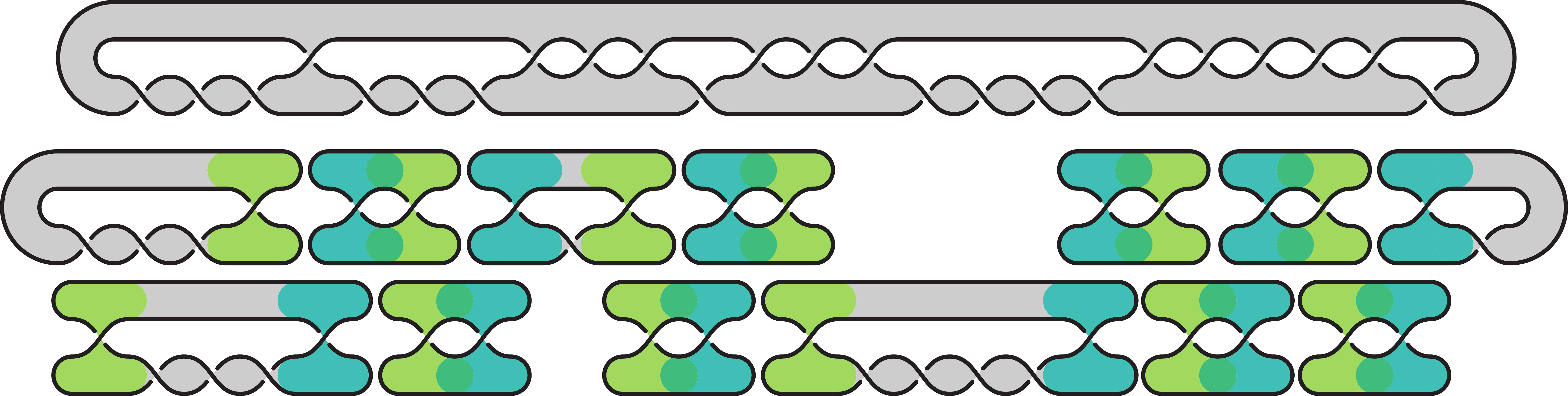}
\caption{The Hatcher--Thurston surface associated to  $[4,5,2,2,3,2,2,6,2,2,2,2,2]_-$}
\label{Fi:PosCB}
\end{center}
\end{figure}

Note that, given $\mathbf b=[b_1,\dots,b_k]_-$, the diagram $D$ we have just described is reduced and alternating, so by Tait's first conjecture \cite{Kauffman1987,Murasugi1987,Thistlethwaite1987} it realizes the crossing number of $L[b_1,\dots,b_k]_-$:
\begin{equation*}\label{E:crossingnumbereqn}
\text{cr}(L[b_1,\dots,b_k]_-)=\sum_{i=1}^k|b_i|-\sum_{i=1}^{k-1}\delta_{\text{sign}(b_i),\text{sign}(b_{i+1})},
\end{equation*}
where $\delta_{x,y}$ equals 1 if $x=y$ and 0 otherwise (we will use Kronecker's delta notation throughout). This is because each crossing in $D$ contributes $\pm\frac{1}{2}$ to the framing(s) of either one or two of the bands $F_i$, and the crossings that contribute to two framings are precisely those where $F_i$ and $F_{i+1}$ are plumbed together when $b_i$ and $b_{i+1}$ have the same sign. In particular, when all $b_i\geq2$, we have
\begin{equation}\label{E:subformcrossingnum}
\text{cr}(L[b_1,\dots, b_k]_-) = \sum_{i=1}^k b_i - (k-1).
\end{equation}

Finally, note that all-even subtractive continued fractions correspond to Seifert surfaces, since the surface $F=F_1*\cdots*F_k$ we have constructed from a subtractive continued fraction $[b_1,\hdots,b_k]_-$ is orientable if and only if every $b_i$ is even.  In particular, $p/q$ will have exactly one all-even subtractive continued fraction if $q$ is odd (since $L(p/q)$ is a knot and exactly one state surface from $D$ is orientable), and $p/q$ will have exactly two all-even subtractive continued fractions if $q$ is even (since $L(p/q)$ is a 2-component link and exactly two state surfaces from $D$ are orientable).

In summary, we have:

\begin{theorem}
Let $p/q$ be a reduced fraction with $0<p<q$, and let $D$ be the reduced alternating diagram of $L(p/q)$ described in this subsection.  Then the adequate state surfaces from $D$ correspond, via the construction described in this subsection, to the subtractive continued fraction representations of $p/q$ and $(p-q)/q$ with all coefficients at least 2 in absolute value.
\end{theorem}

\begin{proof}
Given such a continued fraction representation, we have described how to construct the associated surface.  Conversely, by a result of Ozawa \cite{Ozawa:2011}, each such surface is $\pi_1$-essential, hence, by a result of Hatcher and Thurston \cite{Hatcher-Thurston:1985}, is described by such a continued fraction. 
\end{proof}

\section{A procedure to compute the crosscap number of an alternating link}
\label{sec:algorithm}

In this section, we adapt a procedure of Adams and Kindred to compute the crosscap number of an alternating link.

\subsection{Finding minimal-complexity 1- and 2-sided surfaces}\label{S:Find}

Our formulas for the unoriented genus and crosscap number of a 2-bridge link $L$ come from two results of Adams and Kindred which pertain more generally to alternating knots and links \cite{Adams-Kindred:2013}. The first result tells us that we can compute $\gamma(L)$ and $\Gamma(L)$ by minimizing complexity over all state surfaces for a given alternating diagram of $L$:

\begin{theorem}[Corollary 6.1 of \cite{Adams-Kindred:2013}]\label{T:AK61}
Let $D$ be an alternating diagram of a link $L$. Among all state surfaces from $D$, choose one, $F$, of minimal complexity.  Then $\Gamma(L)=\beta_1(F)$.  If $F$, or any other minimal complexity state surface from $D$, is nonorientable, then $\gamma(L)=\beta_1(F)$.  Otherwise, $\gamma(L)=\beta_1(F)+1$.
\end{theorem}

The second result is that the following procedure narrows the search for the complexity-minimizing state surfaces of a given diagram:

\begin{procedure}[Minimal genus algorithm from \cite{Adams-Kindred:2013}]\label{Proc:AK Alt}~\\
\vspace{-10pt}
  \begin{itemize}
    \item[(1)] For a given link diagram, choose a connected component $D$, and choose a complementary region $R$ of $D$ incident to the fewest crossings (of $D$).
    \item[(2a)] If $R$ abuts one or two crossings, smooth each crossing in $\partial R$ so that $\partial R$ becomes a state circle. 
    \item[(2b)] Otherwise, by a simple Euler characteristic argument, $R$ abuts three crossings. Create two branches of our algorithm, each of which will ultimately yield a state surface. We will later choose to follow the branch that produced the smaller genus surface and ignore the other. For one of these branches, split each crossing on $\partial R$ so that $\partial R$ becomes a circle. For the other branch, split each of these three crossings the opposite way.
    \item[(3)] Repeat until each branch resolves all crossings, yielding a state. Of all associated state surfaces, choose one of minimal complexity.
\end{itemize}  
\end{procedure}

Theorem 3.3 of \cite{Adams-Kindred:2013} states that Procedure~\ref{Proc:AK Alt} always produces a state surface $F$ of minimal complexity (among all state surfaces for a given diagram).  In the alternating case, Theorem~\ref{T:AK61} implies that $F$ realizes the unoriented genus of the link it spans.  The question of determining crosscap number is more subtle.  More on this momentarily.

First, though, we note that Procedure~\ref{Proc:AK Alt} simplifies in the 2-bridge case, due to two observations.  The first is that any nontrivial 2-bridge diagram has bigons or 1-gons, as does any split union of connect sums of such diagrams. The second is that smoothing the crossings around a bigon or 1-gon takes any split union of connect sums of 2-bridge diagrams to another such diagram. Here is the simplified procedure:

\begin{procedure}[Adams--Kindred]\label{Proc:AK 2Bridge}
Let $D$ be a disjoint union of connect sums of alternating 2-bridge link diagrams. In some nontrivial component $D'$ of $D$, choose a 1-gon (if possible) or bigon $X$ of $S^2\setminus D'$, and smooth its incident crossing(s) so that its boundary becomes a state circle.  Repeat until all crossings of $D$ have been resolved, leaving a state $x$ and its state surface $F_x$.
\end{procedure}

Again, Theorem 3.3 and Corollary 6.1 of \cite{Adams-Kindred:2013} imply:

\begin{theorem}\label{T:AKalgo} Given a (disjoint union of connect sums of) 2-bridge alternating diagram(s) $D$ of a link $L\subset S^3$, applying Procedure~\ref{Proc:AK 2Bridge} to $D$ produces a spanning surface $F_x$ of minimal complexity: $\beta_1(F_x)=\Gamma(L)$.
\end{theorem}

We now begin addressing the more subtle question of how to utilize Procedures~\ref{Proc:AK Alt} and \ref{Proc:AK 2Bridge} when computing crosscap number.

\begin{lemma}\label{L:UseDual}
If a link diagram $D\subset S^2$ has a state $x$ such that all cycles in the state graph $G_x$ have length at least 4, then $F_x$ is the unique minimal-complexity state surface from $D$.
\end{lemma}

\begin{proof}
{Let $|D|$ denote the number of components of $D$, and given a state $y$, let $|y|$ denote the number of state circles in $y$.} We argue by induction on $n$, the number of crossings in $D$.  

The base case is vacuous.  For the induction step, {let $D$ and $x$ be as hypothesized, and} let $y$ be a state of $D$ whose state surface has minimal complexity. Then $|y|$ is maximal among all states of $D$; in particular, $|x|\leq |y|$. Assume for contradiction that $y$ and $x$ are different states. 

If $y$ and $x$ are not dual states, then there is a crossing $c$ where their smoothings match. Resolving $c$ (following the smoothing shared by $x$ and $y$) reduces $D$ to an $(n-1)$-crossing diagram $D'$ with states $x'$ and $y'$ identical to $x$ and $y$ but without the crossing arc at $c$.  In particular, all cycles in the state graph $G_{x'}$ have length at least 4, so $|y'|<|x'|$ by induction. Yet, $|x'|=|x|$ and $|y'|=|y|$.  Contradiction. 

We may thus assume that $x$ and $y$ are dual. It follows that each state circle in $y$ corresponds to a cycle in $G_x$ and therefore is incident to at least four crossings. Meanwhile, each crossing is incident to at most two state circles of $y$, so $|y|\leq n/2$.  

For each component $D_i$ of $D$, say with $n_i$ crossings, denote its checkerboard states by $b_i$ and $w_i$, so that $|b_i|+|w_i|=n_i+2$.  Let $b$ and $w$, respectively, be the disjoint union of the $b_i$'s and $w_i$'s. 
Summing over components of $D$ gives $|b|+|w|=n+2|D|$.   Yet $|b|\leq|y|\leq n/2$ and $|w|\leq|y|\leq n/2$, which is a contradiction. 
 \end{proof}

The following lemma extends Theorem 5 of Hirasawa and Teragaito \cite{Hirasawa_Teragaito:2006} to alternating links.

\begin{lemma}\label{L:HopfBand}
Let $D$ be an alternating diagram of a nontrivial link $L\subset S^3$ with no Hopf link connect summand, and let $F_x$ be a state surface from $D$ with $\beta_1(F_x)=\Gamma(L)$.  If $\gamma(L)=\Gamma(L)+1$, then every cycle in the state graph $G_x$ has an even length of at least 4.
\end{lemma}

\begin{proof}
We will prove the contrapositive: if $G_x$ has an odd cycle or a cycle of length 2, then $\gamma(L)=\Gamma(L)$.  If $G_x$ has an odd cycle, then $F_x$ is 1-sided, so $\gamma(L)\leq \beta_1(F_x)=\Gamma(L)\leq \gamma(L)$, hence $\gamma(L)=\Gamma(L)$.

Assume instead that $G_x$ has no odd cycles (so $F_x$ is 2-sided) and contains a cycle of length 2.  The edges in this cycle correspond to crossings $c$ and $c'$, and the crossing bands at $c$ and $c'$ in $F_x$ join the same pair of (disks bounded by) state circles.  Let $y$ denote the state of $D$ that is identical to $x$, except at $c$ and $c'$ where it has the opposite smoothings, and let $F_y$ denote its state surface.  Observe that $x$ and $y$ have the same number of state circles, so $\beta_1(F_x)=\beta_1(F_y)$. See Figure~\ref{Fi:Hopfmobius}.

\begin{figure}
\begin{center}
\includegraphics[width=.6\textwidth]{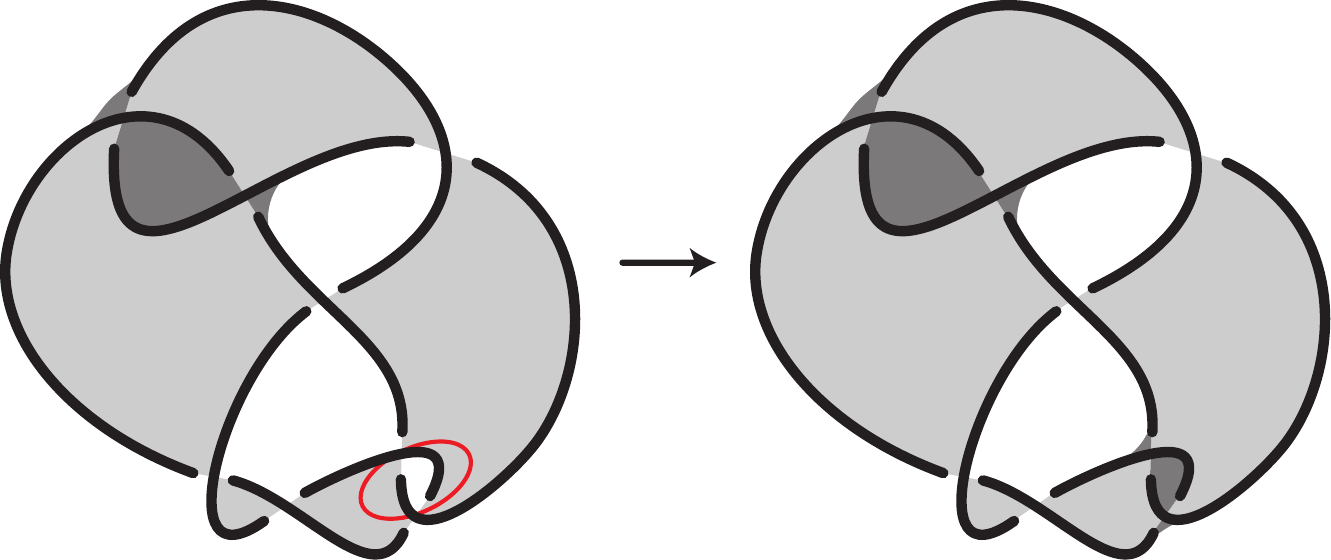}
\caption{If the state graph $G_x$ of a 2-sided state surface $F_x$ (with no Hopf band boundary connect summands) has a cycle of length two,  then reversing two smoothings gives a 1-sided state surface $F_y$ with $\beta_1(F_y)=\beta_1(F_x)$.}
\label{Fi:Hopfmobius}
\end{center}
\end{figure}

Note that the surface obtained from $F_x$ by cutting through these two crossing bands is connected, or else $L$ would have a Hopf link connect summand.  Therefore, there is a simple closed curve $\gamma$ in $F_x$ that passes exactly once through one of the two crossing bands in question and is disjoint from the other crossing band. There is an associated simple closed curve $\gamma'$ in $F_y$.  The neighborhood of $\gamma$ in $F_x$ is annular because $F_x$ is 2-sided, and the framing of $\gamma'$ in $F_y$ differs from that of $\gamma$ in $F_x$ by $\pm\frac{1}{2}$, and so $F_y$ is 1-sided, giving:
\[\pushQED\qed
\gamma(L)\leq \beta_1(F_y)=\beta_1(F_x)=\Gamma(L)\leq \gamma(L).\qedhere\]
\end{proof}

We now have the tools we need to prove our first main result.

\begin{proof}[Proof of Theorem \ref{T:CC}]
One direction is precisely Lemma~\ref{L:HopfBand}.  Conversely, if every cycle in $G_x$ has an even length of at least 4, then $F_x$ is the unique minimal-complexity state surface from $D$ by Lemma~\ref{L:UseDual}, and $F_x$ is 2-sided, so $\gamma(L)=\Gamma(L)+1$ by Theorem \ref{T:AK61}.
\end{proof}

Lemmas~\ref{L:UseDual} and \ref{L:HopfBand} similarly give the following characterization:

\begin{theorem}\label{T:2BridgeCCNew}
Suppose that a diagram $D$ of a 2-bridge link $L$ corresponds to the all-even subtractive continued fraction $[e_1,
\hdots, e_k]_-$, and assume that $D$ has more than two crossings. If $L$ has two components, assume further that, between the two all-even subtractive continued fractions for $L$, $[e_1,
\hdots, e_k]_-$ has minimal length. 
Then the following statements are equivalent:
\begin{enumerate}
\item $D$ has a unique minimal complexity state surface $F_x$, and $F_x$ is 2-sided.
\item $\gamma(L)= \Gamma(L)+1$.
\item Each $|e_i|\geq 4$.
\end{enumerate}
\end{theorem}

\begin{proof}
{Note that $D$ is alternating, so (1) and (2) are equivalent by Theorem~\ref{T:AK61}. Note also that $L$ is prime. Assume that $D$ has more than two crossings (otherwise $L$ is a Hopf link and all three statements are false).  For the purpose of using Lemmas~\ref{L:UseDual} and \ref{L:HopfBand}, let $F_x$ be the Hatcher-Thurston surface associated to  $[e_1,
\hdots, e_k]_-$, and let $G_x$ be its state graph.  If (3) holds, then Lemma~\ref{L:UseDual} implies (1). If (2) holds, then Lemma~\ref{L:HopfBand} implies (3).}
\end{proof}

\subsection{Formulas for unoriented genus and crosscap number}\label{S:Formulas}
{The following two lemmas describe how the unoriented genus of a 2-bridge link 
{is affected by} 
small changes to the continued fraction describing the link.}
We write $\Gamma[b_1,\ldots,b_k]_-$ for the unoriented genus $\Gamma(L[b_1,\ldots,b_k]_-)$ of the associated link.

\begin{lemma}\label{L:Rule1}
If  $D$ corresponds to the positive subtractive continued fraction {$[b_1,\hdots, b_k]_-$ with 
$b_k\geq 3$, then 
\[\Gamma[b_1,\hdots,b_{k}]_-=1+\Gamma[b_1,\hdots,b_{k-1}]_-.\]}
\end{lemma}
\begin{proof}
{Because $b_k\geq 3$, there is a bigon at the (right) end}
of $D$ with which we can begin Procedure~\ref{Proc:AK 2Bridge}. Smoothing around this bigon and any resulting 1-gons resolves {$b_k-1$} crossings and gives {$b_k-2$} state circles together with {the diagram that corresponds (after a planar isotopy)} to {$[b_1,\hdots,b_{k-1}]$}. See the top row of Figure~\ref{Fi:Rule13}.
\end{proof}

\begin{figure}
\begin{center}
\includegraphics[height=.64in]{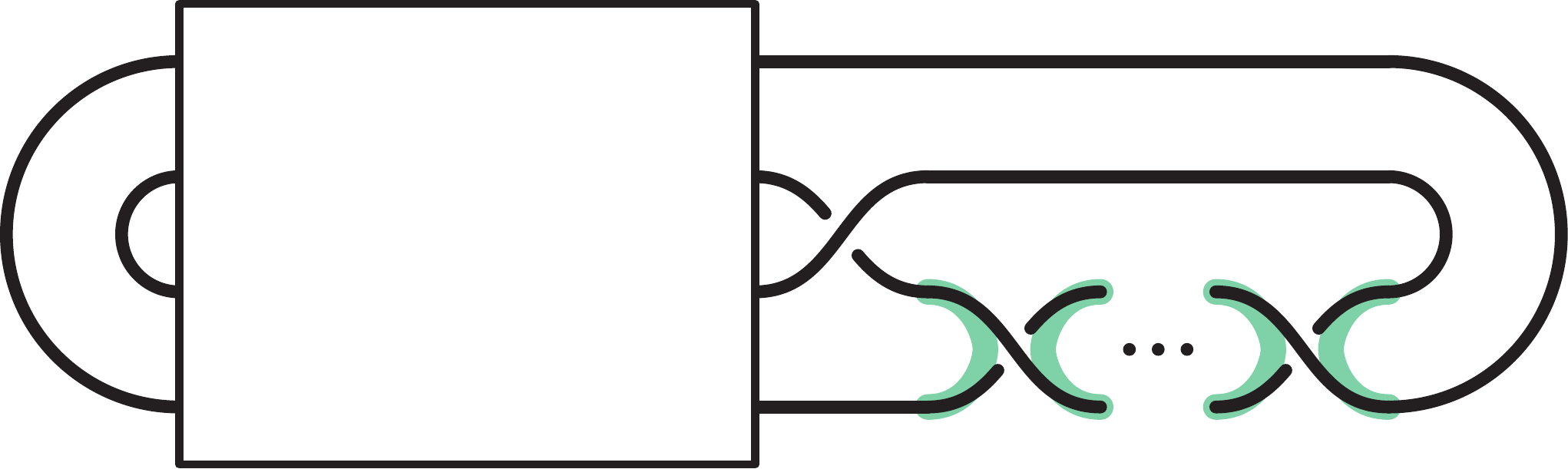}~\raisebox{.32in}{$\to$}~\includegraphics[height=.64in]{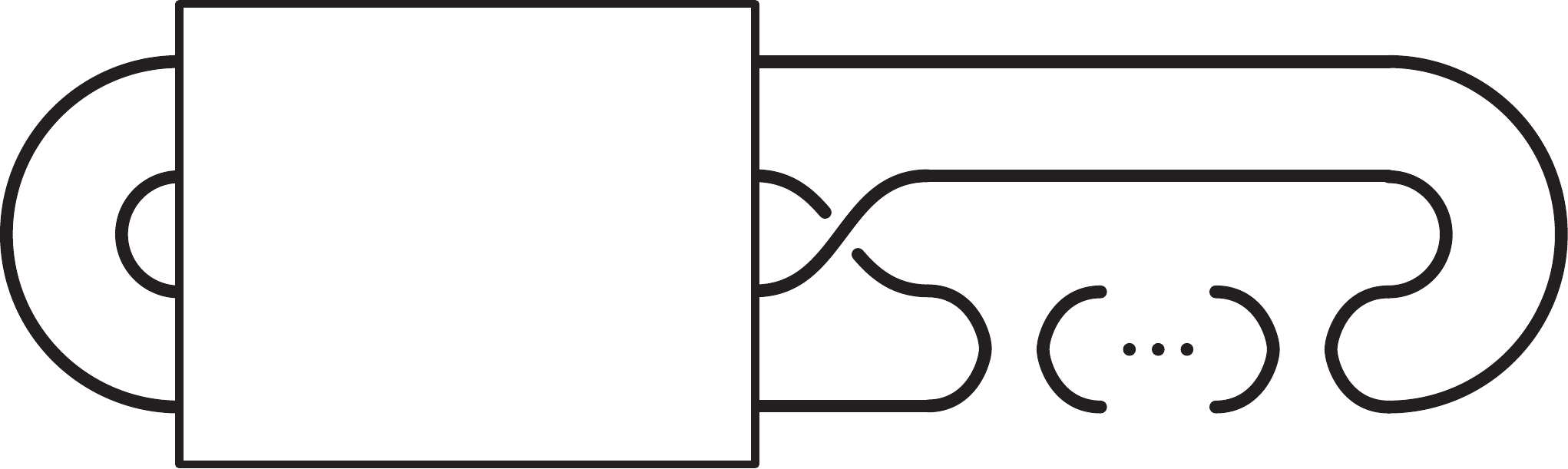}~\raisebox{.32in}{$\to$}~\includegraphics[height=.64in]{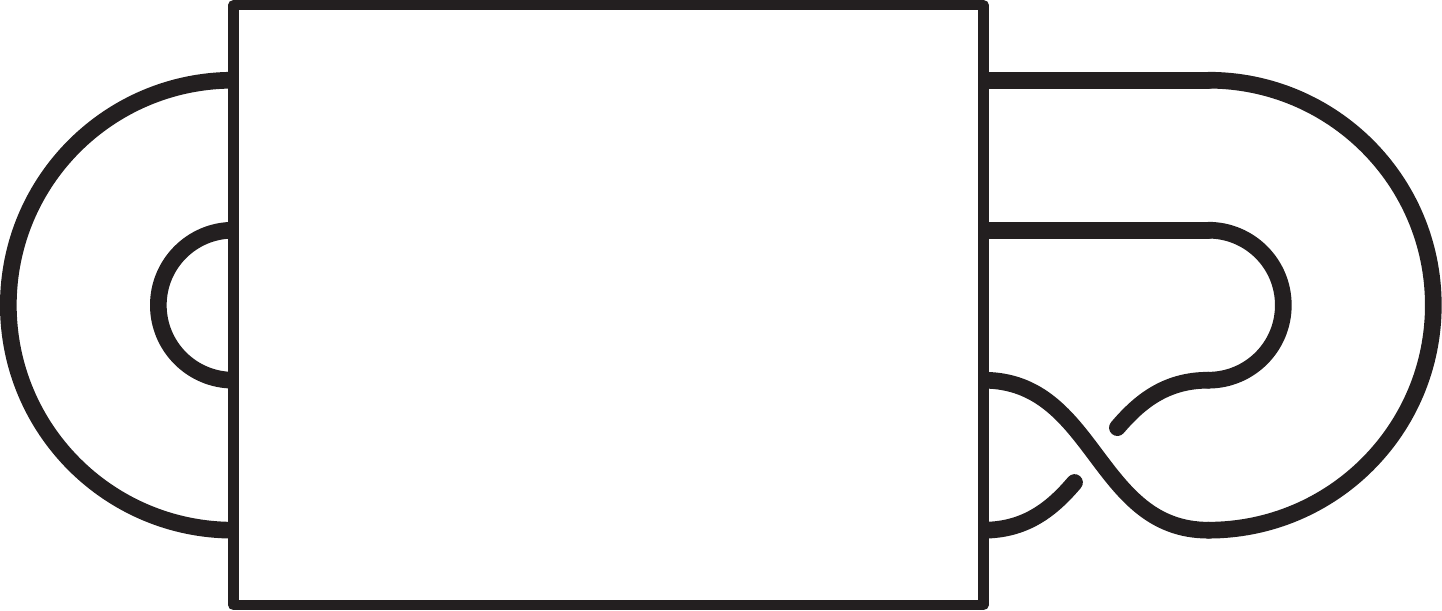}\\
\includegraphics[height=.64in]{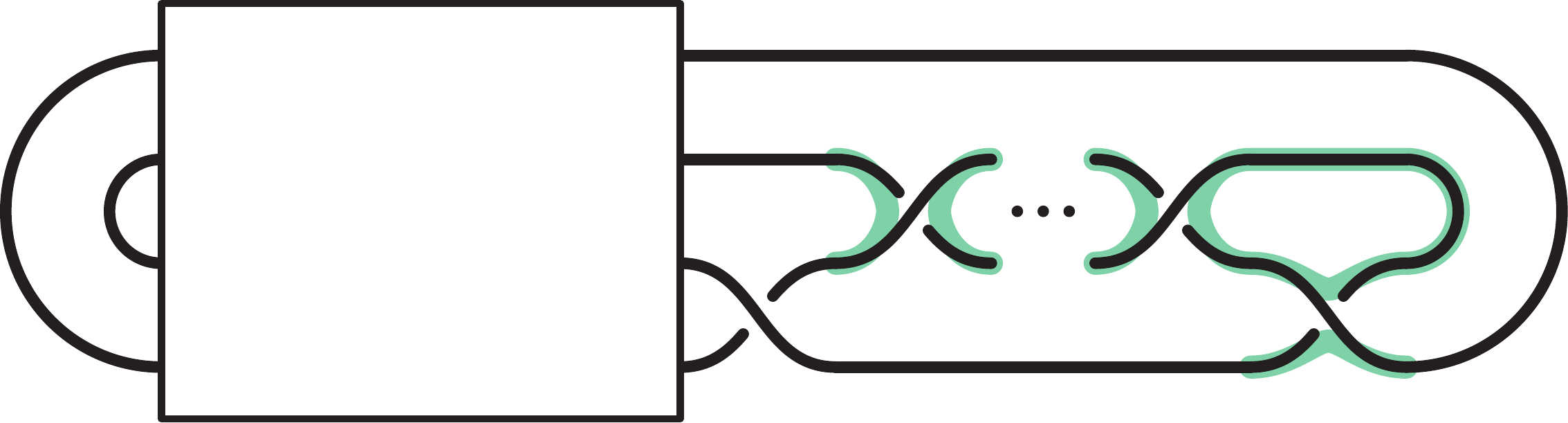}~\raisebox{.32in}{$\to$}~\includegraphics[height=.64in]{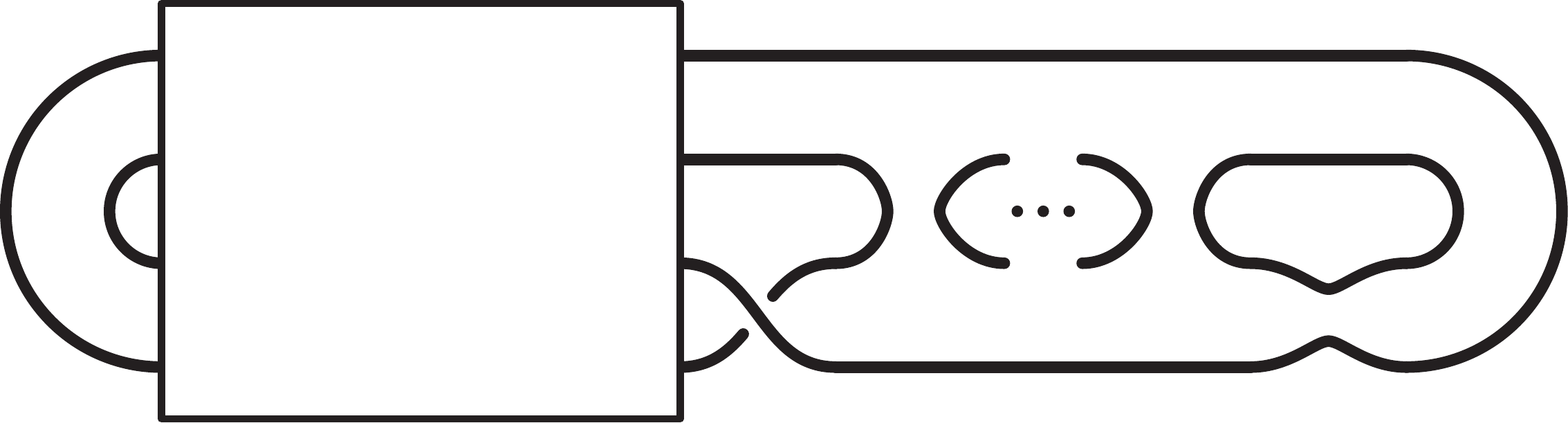}\\
\caption{The situations in Lemmas~\ref{L:Rule1} and \ref{L:Rule3}.
}
\label{Fi:Rule13}
\end{center}
\end{figure}

\begin{lemma}\label{L:Rule3}
Suppose $D$ corresponds to the positive subtractive continued fraction {$[b_1,\hdots, b_k]_-$}.  If for some $j$ we have $b_j\geq3$ 
and $b_i=2$  for all {$j<i\leq k$}, then
\[{\Gamma[b_1,\hdots,b_k]_-=\Gamma[b_1,\hdots,b_j,2,\hdots,2]_-=1+\Gamma[b_1,\hdots,b_{j-1},b_{j}-1]_-.}\]
\end{lemma}

\begin{proof}
Resolving the {$k-j$} bigons at the {(right) end} of the diagram involves smoothing {$k-j+1$} crossings and yields {$k-j$} state circles and the diagram corresponding to 
{$[b_1,\hdots,b_{j-1},b_{j}-1]_-$}. See the bottom row of  Figure~\ref{Fi:Rule13}.
\end{proof}

{Lemmas~\ref{L:Rule1} and \ref{L:Rule3} imply the following theorem.}

\begin{theorem}\label{T:2BridgeGamma} If a diagram $D$ of a 2-bridge link $L$ corresponds to the positive subtractive continued fraction $[b_1,
\hdots, b_k]_-$
, then $\Gamma(L)=w-z$, where:
\begin{itemize}
\item $w=\#\{i~:b_i=2\text{ and either }i=1\text{ or }b_{i-1}\neq2\}+\#\{i:~b_i\geq3\}$ and 
\item $z=\#\{i:~\text{for some }j\geq 0,~b_i=3=b_{i+1}=\cdots=b_{i+j}\text{ and }b_{i-1}=2=b_{i+j+1}\}$. 
\end{itemize}
Note that the first quantity in $w$ counts the number of strings of 2's in $[b_1,\hdots,b_k]_-$, and that $z$ counts the number of strings of 3's in $[b_1,\hdots,b_k]_-$ that are immediately preceded {\it and} succeeded by 2's.
\end{theorem}

\begin{proof}
We argue by (strong) induction on $w$.  If $w=1$, then 
$L$ is a $(2,q)$ torus link, hence bounds either an annulus or a M\"obius band. Also, $z=0$, so $\Gamma(L)=1=w-z$.

For the induction step, consider the {largest} $i$ for which $b_i\geq 3$.  If $i=k$, so $b_k\geq 3$, then observe that 
{$z[b_1,\hdots,b_k]=z[b_1,\hdots,b_{k-1}]$ and $w[b_1,\hdots,b_k]=w[b_1,\hdots,b_{k-1}]+1$.}
Thus, by Lemma~\ref{L:Rule1} and induction:
\begin{equation*}\begin{split}
\Gamma[b_1,\hdots,b_k]_-&=1+\Gamma{[b_1,\hdots,b_{k-1}]_-}\\
&=1+w{[b_1,\hdots,b_{k-1}]}-z{[b_1,\hdots,b_{k-1}]}\\
&=w[b_1,\hdots,b_k]-z[b_1,\hdots,b_k].
\end{split}
\end{equation*}
Otherwise, we have {$b_i>2$ and $b_{i+1}=\cdots=b_{k}=2$ for some $i<k$}. 
We split into {three cases. First, if $i=1$ and $b_i=3$, then $w[b_1,\hdots,b_k]=2$, $z[b_1,\hdots,b_k]=0$, and by Lemma~\ref{L:Rule3},} 
\begin{equation*}
{\Gamma[b_1,\hdots,b_k]_-=\Gamma[3,2\hdots,2]_-=1+\Gamma[2]_-=2=w[b_1,\hdots,b_k]-z[b_1,\hdots,b_k].}
\end{equation*}
{Second}, if $b_i=3$ and {$b_{i-1}=2$}, then $w[b_1,\hdots,b_k]=2+w{[b_1,\hdots,b_{i-1},b_i-1]}$ and $z[b_1,\hdots,b_k]=1+z{[b_1,\hdots,b_{i-1},b_i-1]}$, so Lemma~\ref{L:Rule3} and induction give
\begin{equation}\label{E:GammaXYZ}\begin{split}
\Gamma[b_1,\hdots,b_k]_-&=1+\Gamma{[b_1,\hdots,b_{i-1},b_i-1]_-}\\
&=1+w{[b_1,\hdots,b_{i-1},b_i-1]}-z{[b_1,\hdots,b_{i-1},b_i-1]}\\
&=w[b_1,\hdots,b_k]-z[b_1,\hdots,b_k].
\end{split}
\end{equation}

Otherwise, we have either $b_i=3$ and {$b_{i-1}>2$}, or $b_i\geq 4$.
{In both cases, one carefully checks that} $w[b_1,\hdots,b_k]=1+w{[b_1,\hdots,b_{i-1},b_i-1]}$ and $z[b_1,\hdots,b_k]=z{[b_1,\hdots,b_{i-1},b_i-1]}$, so Lemma~\ref{L:Rule3}, induction, and Equation~\eqref{E:GammaXYZ} complete the proof.
\end{proof}

Theorems~\ref{T:2BridgeCCNew} and \ref{T:2BridgeGamma}
immediately imply Theorem~\ref{thm:cc_formula}.

\section{Computing unoriented genus}
\label{sec:unoriented_genus}

In this section we use Theorem~\ref{T:2BridgeGamma} to determine formulas for the total unoriented genus of 2-bridge knots with a fixed crossing number $c$. Since there are several technical details that obscure the derivation, we begin with a brief outline of the approach.

\subsection{Overview}

If $p/q=[b_1,\dots,b_k]_-$ and $q$ is odd, the link $L[b_1,\dots,b_k]_-$ is a knot that we denote by $K[\bf b]_-$.  Let $\mathcal K_c$ be the set of all 2-bridge knots with crossing number $c$ and where mirror images are counted as distinct knots unless they are isotopic and let $K(c)$ be the set of all tuples or vectors $\mathbf b$ with integer entries $b_i \ge 2$ that give a knot in $\mathcal K_c$ via the identification $\mathbf b \mapsto K[\bf b]_{-}$. From Equation\eqref{E:subformcrossingnum} it follows that
\begin{equation*}\label{E:defnK(c)}
K(c) = \left\{(b_1,
\dots,b_k): b_i \ge 2,\ K[\mathbf b]_- \mbox{ is a knot, and } c = \sum_{i=1}^k b_i - (k-1) \right\}. 
\end{equation*}

Most knots are represented by two vectors $(b_1, \dots , b_k)$ and $(b_k, \dots , b_1)$ in $K(c)$, 
although knots represented by palindromic tuples are represented by exactly one vector. 
From this it follows that
\[
|\mathcal K_c| = \frac{1}{2} \left(|K(c)| + |K^P(c)|\right)
\]
where $K^P(c) \subset K(c)$ is the subset of palindromic vectors that represent knots.
From Theorem~\ref{T:2BridgeGamma}, we have $\Gamma(K[ \mathbf b]_-) = w([\mathbf b]_-) - z([\mathbf b]_-)$ for any $\mathbf b \in K(c)$. 
Moreover, because the values of $w$ and $z$ are identical for $(b_1,\dots, b_k)$ and $(b_k, \dots, b_1)$, we will abuse notation and write $w(K)$ or $z(K)$ for $K \in \mathcal K_c$, avoiding reference to the choice of vector $\mathbf b$. Using this notation, the average unoriented genus for 2-bridge knots with crossing number $c$ is given by
\[
\overline{ \Gamma}(c) = \frac{1}{|\mathcal K_c|}\left(\sum_{K \in \mathcal K_c} (w(K)-z(K)) \right).
\]

The denominator in the average unoriented genus is given by the formula for $|\mathcal K_c|$ produced by Ernst and Sumners \cite{Ernst-Sumners:1987} given below.

\begin{theorem}[Ernst-Sumners] 
\label{T:Ernst-Sumners} The number of distinct 2-bridge knots with crossing number $c$ is given by
\begin{equation}\label{E:Ernst-Sumners}
|\mathcal K_c| = \left\{ \begin{array}{ll}
\displaystyle \frac{2^{c-2}-1}{3}, & \mbox{if $c$ even and $c \ge 4$,} \\
\\
\displaystyle \frac{2^{c-2}+2^{(c-1)/2}}{3}, & \mbox{if $c \equiv 1 \Mod{4}$ and $c \ge 5$,} \\
\\
\displaystyle \frac{2^{c-2}+2^{(c-1)/2}+2}{3}, & \mbox{if $c \equiv 3 \Mod{4}$ and $c \ge 7$.}
\end{array}
\right.
\end{equation}
\end{theorem}

Therefore, we need only determine the appropriate formula for $\sum_{K \in \mathcal K_c} (w(K)-z(K))$ in order to calculate the average unoriented genus.
If we let 
\begin{equation}\label{E:WZ}
\begin{array}{lcl}
\displaystyle W(c) = \sum_{\mathbf b \in K(c)}w(K[\mathbf b]_-), & &  \displaystyle Z(c) = \sum_{\mathbf b \in K(c)} z(K[\mathbf b]_-), \\
\displaystyle W^P(c) = \sum_{\mathbf b \in K^P(c)} w(K[\mathbf b]_-), & & \displaystyle Z^P(c) = \sum_{\mathbf b \in K^P(c)} z(K[\mathbf b]_-),
\end{array}
\end{equation}
then
\[
\sum_{K \in \mathcal K_c} (w(K)-z(K)) = \frac{1}{2} \left( W(c) - Z(c) + W^P(c) - Z^P(c) \right). 
\]
In this section, we find explicit formulas for $W(c)$ and $Z(c)$ (Proposition~\ref{P:WZforumlas}), and $W^P(c)$ and $Z^P(c)$ (Proposition~\ref{P:WpZpformulas}). These formulas combined with Theorem~\ref{T:Ernst-Sumners} then provide the following closed formula for $\overline \Gamma(c)$, establishing the first part of Theorem~\ref{T:ccave}.

\begin{theorem}\label{T:Ep1}
Let $c \ge 11$ and set $d = \frac{c-1}{2}$. The average unoriented genus $\overline \Gamma(c)$ of a 2-bridge knot with crossing number $c$ is
\[
\overline \Gamma(c) = \frac{c}{3} + \frac{1}{9} + \varepsilon_1(c)
\]
where
\[
\varepsilon_1(c) = \left\{ \begin{array}{ll}
\displaystyle \frac{c-2 + 3\, \delta_{1,c~\mbox{\scriptsize mod}~3}}{3(2^{c-2}-1)},& \mbox{if $c$ even},\\
\\
\displaystyle \frac{(6d+3)2^{d+1} -4 -18 \,\delta_{2,d~\mbox{\scriptsize mod}~3}}{9(2^{c-2}-2^d)},& \mbox{if $c \equiv 1 \Mod{4}$}, \\
\\
\displaystyle \frac{(6d+3)2^{d+1} -6c-2 + 18\,  (\delta_{1,d~\mbox{\scriptsize mod}~3}+2\, \delta_{2,d~\mbox{\scriptsize mod}~3})}{9(2^{c-2}-2^d+2)},& \mbox{if $c \equiv 3 \Mod{4}$}.
\end{array}
\right.
\]
Since $\varepsilon_1(c) \rightarrow 0$ as $c \rightarrow \infty$, the average unoriented genus approaches $\frac{c}{3} + \frac{1}{9}$ as $c \rightarrow \infty$.
\end{theorem}

The closed formulas for $W(c)$, $Z(c)$, $W^P(c)$, and $Z^P(c)$ are all derived in a similar manner. 
There are three main steps, which we describe now in the case of $W(c)$. 
\begin{enumerate} 
\item First, in Lemma~\ref{L:fbijections}, we show that the set of vectors $K(c)$ can be defined recursively. To do this, we partition $K(c)$ into four subsets 
\[
K(c) = K_{22}(c) \sqcup K_{\bar32}(c) \sqcup K_{3}(c) \sqcup K_{\bar 4}(c),
\]
defined below in Equation \eqref{eq:Kt}, and establish a bijection 
\[
g = f_1 \cup f_2 \cup f_3 \cup f_4 : K(c) \rightarrow K(c-2) \sqcup K(c-1) \sqcup K(c-2).
\]
\item Next, we use the bijection $g$ to determine a recursive description for $W(c)$. For many vectors, the bijection does not change the value of $w$, but for some particular subsets of vectors the value of $w$ decreases by one or two (see Lemma~\ref{L:fwzchange}). Thus, the bijection $g$ implies that $W(c) = W(c-1) + 2 W(c-2) + D_w$ where $D_w$ is the number of vectors where $g$ decreases the value of $w$ by one plus twice the number of vectors where $g$ decreases $w$ by two. In Lemma~\ref{L:Kcardinalities} we determine the sizes of the subsets of $K(c)$ for which $w$ changes, and from this we establish the recursion
\[
W(c) = W(c-1) + 2 W(c-2) + 3 \cdot 2^{c-5}.
\]
\item Finally, we solve the recursion subject to the initial values $W(4)=4$ and $W(5)=10$. Because the homogeneous recursion $W(c) = W(c-1) + 2 W(c-2)$ has a general solution $u_0(2^{c}) + v_0(-1)^c$, we solve the non-homogeneous recursion for $W(c)$ by determining appropriate values of coefficients in a solution of the form $(u_0+u_1 c)(2^{c}) + v_0(-1)^c$. This strategy provides the closed form $W(c) = c \cdot 2^{c-4}$ for $c \ge 4$ in Proposition~\ref{P:WZforumlas}. 
\end{enumerate}
\begin{rem}
While the formula suggests that perhaps some simpler approach might have worked, our approach has the virtue that it also yields formulas for $Z(c)$, $W^P(c)$, and $Z^P(c)$.
\end{rem}

\subsection{The asymptotic formula for average unoriented genus}

As mentioned above, the set of vectors $K(c)$ that define 2-bridge knots via the identification $\mathbf b \mapsto K[\mathbf b]_-$ can be defined recursively by a bijection $g$. We first establish some notation to help define this function. Let
 \begin{equation}
 \label{eq:Kt}
K_{\mathbf t}(c) = \left\{\mathbf b \in K(c) : \mathbf b\ \mbox{ends with the sub-vector} \ \mathbf t \right\},
\end{equation}
where omitting $\mathbf t$ means we allow any ending for $\mathbf b$ and  where $b_i = \bar s$ in $\mathbf t$ means $b_i \ge s$. For strings of the same value $b$ in $\mathbf t$ we shall use the notation $b^{[m \ge 1]}$ to refer to all vectors in $K(c)$ that have a substring of $b$'s with length $m \ge 1$. For example, 
\[
K_{23^{[m \ge 1]}22}(c) = K_{2322}(c) \sqcup K_{23322}(c) \sqcup K_{233322}(c) \sqcup \dots
\]
Of course, only finitely many of the sets in this union are nonempty because of the constraint $c = \sum_{i=1}^k b_i - (k-1)$. Finally, for any vector $\mathbf b$, we also define
\[
K_{\{\mathbf b\}}(c)= K(c) \cap \{\mathbf b\} 
\]
and note that $K_{\{\mathbf b\}}(c)$ may be empty because the vector $\mathbf{b}$ may correspond to a 2-component link rather than a knot. For example, $K_{\{(4,2^{[n \ge 0]})\}}(6) = \varnothing$ since $c=6$ implies $n = 2$, but $(4,2,2)$ represents a 2-component link instead of a knot.

For $c\geq 5$, partition $K(c)=K_{22}(c)\sqcup K_{\bar 32}(c)\sqcup  K_{3}(c)\sqcup K_{\bar 4}(c)$ and define four functions \begin{align*}
f_1 :&\, K_{22}(c) \rightarrow K(c-2) \\
f_2 :&\, K_{\bar 32}(c) \rightarrow K_{2}(c-1) \\
f_3 :&\, K_{3}(c) \rightarrow K_{\bar 3}(c-1) \\
f_4 :&\, K_{\bar 4}(c) \rightarrow K(c-2)
\end{align*} 
as shown in Figure~\ref{Fi:f1f2f3f4}, namely 
\[
\arraycolsep=1.4pt\def\arraystretch{1}
\begin{array}{lll}
f_1(b_1,
\dots ,b_{k-2} ,2,2) &  =  (b_1, 
\dots, b_{k-2}), & \mbox{for $k \ge 3$}, \\
f_2(b_1,
\dots,b_{k-1},2) & =  (b_1, 
\dots, b_{k-1}-1,2), & \mbox{for $k \ge 2$ and $b_{k-1} \ge 3$}, \\
f_3(b_1,
\dots,b_{k-1},3) & =  (b_1, 
\dots, b_{k-1}+1),&  \mbox{for $k \ge 2$}, \\
f_4(b_1,
\dots,b_{k-1},b_k) & =  (b_1, 
\dots, b_{k-1},b_k-2),& \mbox{for $b_k \ge 4$}.
\end{array}
\]
These bijections determine a Jacobsthal recursion on 2-bridge knots and were motivated by a similar approach taken in \cite{CohLow}.

\begin{figure}[h!]
\begin{center}
\labellist\tiny\hair 4pt
\pinlabel{$f_1$} at 125 300
\pinlabel{$f_2$} at 630 300
\pinlabel{$f_3$} at 1135 300
\pinlabel{$f_4$} at 1640 300
\pinlabel{$f_1^{-1}$} at 255 300
\pinlabel{$f_2^{-1}$} at 770 300
\pinlabel{$f_3^{-1}$} at 1275 300
\pinlabel{$f_4^{-1}$} at 1780 300
\endlabellist
\includegraphics[width=.75\textwidth]{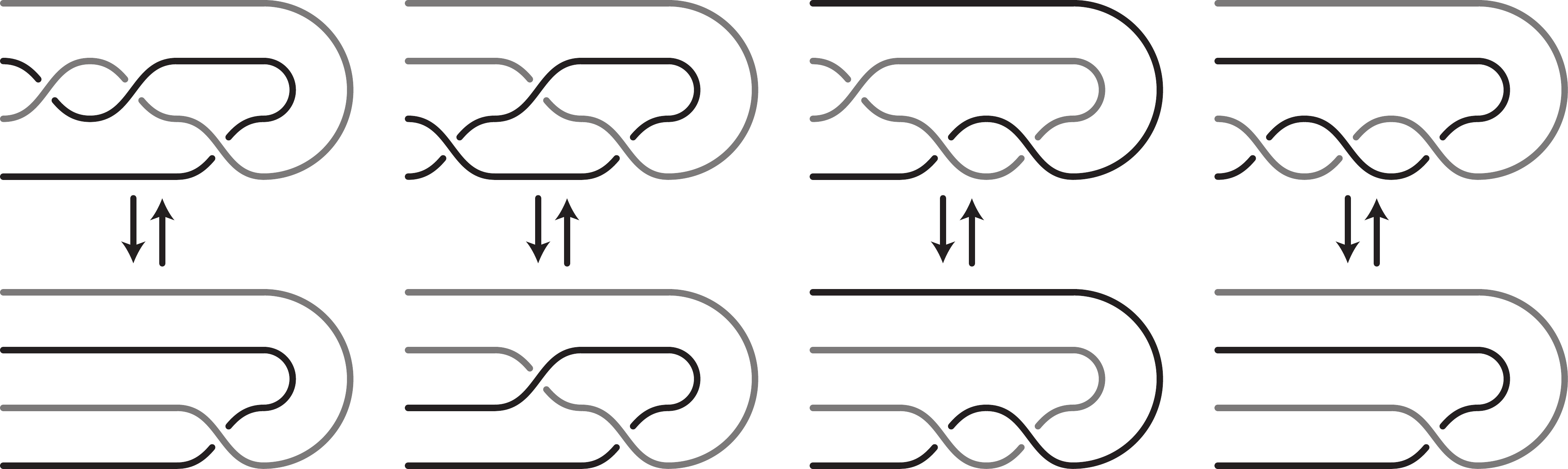}
\caption{The bijections $f_1$, $f_2$, $f_3$, and $f_4$ }
\label{Fi:f1f2f3f4}
\end{center}
\end{figure}
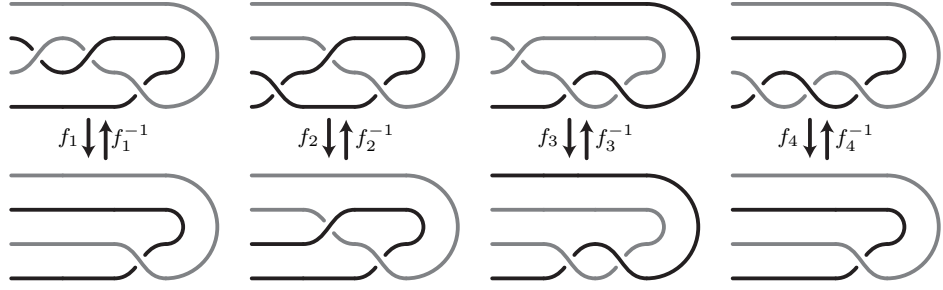

\begin{lemma}\label{L:fbijections}
The functions $f_1$, $f_2$, $f_3$, and $f_4$ defined above are bijections.
Hence, for $c \ge 5$, the function $g=f_1\cup f_2\cup f_3\cup f_4 : K(c) \rightarrow  K(c-2)\sqcup K(c-1)  \sqcup K(c-2)$
 is a bijection.
\end{lemma}

\begin{proof}
It is evident from the shading in Figure~\ref{Fi:f1f2f3f4} that each function $f_i$ takes knots to knots, and that $f_2$ and $f_3$ each decrease $c$ by one, while $f_1$ and $f_4$ each decrease $c$ by two. Thus, each function has the advertised domain and codomain.  Bijectivity follows from the existence of the following inverse functions, which can also be seen in Figure~\ref{Fi:f1f2f3f4}:

\[
\arraycolsep=1.4pt\def\arraystretch{1}
\begin{array}{lll}
f_1^{-1}(b_1, 
\dots, b_{k},2)&= (b_1,
\dots,b_{k}+1,2), \\
f_2^{-1}(b_1, 
\dots, b_{k}) &=(b_1,
\dots ,b_{k} ,2,2), \\
f_3^{-1}(b_1, 
\dots, b_{k})&= (b_1,
\dots,b_{k}-1,3),\\
f_4^{-1}(b_1, 
\dots, b_{k-1},b_k) &= (b_1, 
\dots, b_{k-1},b_k+2)
.
\end{array}
\]

\end{proof}

\begin{example}
\label{Ex:K6example}
The bijection $g:K(6) \rightarrow K(4) \sqcup K(5) \sqcup K(4)$ is shown explicitly in Table~\ref{T:K6bijection}. On the left the partition $K(6) = K_{22}(6) \sqcup K_{\bar 32}(6) \sqcup K_3(6) \sqcup K_{\bar 4}(6)$ is depicted in four boxes. We chose to order the vectors lexicographically from the end (right). The images $f_1(K_{22}(6))=K(4)$ and $f_4(K_{\bar 4}(6))=K(4)$ are depicted as the topmost and bottommost boxes on the right. The image $f_2 \cup f_3 (K_{\bar 32}(6) \sqcup K_3(6)) = K(5)$ is depicted in the middle box on the right. Notice that $f_2 \cup f_3$ induces the partition $K(5) = K_{2}(5) \sqcup K_{\bar 3}(5)$ on the image.

\begin{table}[h]
\[
\begin{array}{c|r|c|r|c} \cline{2-2}\cline{4-4}
\multirow{2}{3em}{$K_{22}(6)$} & (3,2,2,2) & \multirow{2}{2em}{$\stackrel{f_1}{\longrightarrow}$} & (3,2) & \multirow{2}{4em}{$K(4)$} \\
& (2,3,2,2) & & (2,3) & \\ [2pt] \cline{2-2} \cline{4-4} 
\multirow{3}{3em}{$K_{\bar 32}(6)$} & (2,2,3,2) & \multirow{3}{2em}{$\stackrel{f_2}{\longrightarrow}$} & (2,2,2,2) & \multirow{6}{4em}{$K(5)$}\\
 & (3,3,2) &  & (3,2,2) & \\
 & (5,2) & & (4,2) &  \\ [2pt]\cline{2-2} \arrayrulecolor{lightgray}
 \cline{4-4} \arrayrulecolor{black}
\multirow{3}{3em}{$K_{3}(6)$} & (2,2,2,3) & \multirow{3}{2em}{$\stackrel{f_3}{\longrightarrow}$} & (2,2,3) & \\
& (2,3,3) & & (2,4) & \\
& (4,3) & & (5) & \\ [2pt] \cline{2-2 } \cline{4-4}
\multirow{2}{3em}{$K_{\bar 4}(6)$} & (3,4) & \multirow{2}{2em}{$\stackrel{f_4}{\longrightarrow}$} & (3,2) & \multirow{2}{4em}{$K(4)$} \\
& (2,5) & & (2,3) & \\ \cline{2-2} \cline{4-4}
\end{array}
\]
\caption{The bijection $g:K(6) \rightarrow K(4) \sqcup K(5) \sqcup K(4)$}
\label{T:K6bijection}
\end{table}

\end{example}

In order to define $W(c)$ and $Z(c)$ recursively, we need to determine how the bijections $f_1$, $f_2$, $f_3$, and $f_4$ change the values of $w$ and $z$. This is found in the following lemma.

\begin{lemma}\label{L:fwzchange}
Given $\mathbf b \in \mbox{domain}(f_j)$ for $1 \le j \le 4$, let $\Delta w = w (\mathbf b) - w (f_j(\mathbf b))$ and $\Delta z = z (\mathbf b) - z (f_j(\mathbf b))$. Then $\Delta w$ and $\Delta z$ are given in the table below where $n \ge 0$ and $m \ge 1$ are fixed integers.
For any vector in the domain of $f_j$ that is not listed in the table we have $\Delta w = 0$ or $\Delta z = 0$.

\[
\renewcommand{\arraystretch}{1.25}
\begin{array}{c|r|r|c}
f_j & \multicolumn{1}{c|}{\mathbf b}  & \multicolumn{1}{c|}{f_j({\mathbf b})} & \Delta w \\ \hline
f_1 & (\dots, \bar 3,2,2) &  (\dots, \bar 3) & 1  \\
f_2 & (\dots, 2,3,2) &  (\dots, 2,2,2) & 2  \\
f_2 & (\dots, \bar 3,3,2) &  (\dots, \bar 3,2,2) & 1  \\
f_3 & (\dots, \bar 3,2,3) &  (\dots, \bar 3,3) & 1  \\
f_3 & (\dots, \bar 3,3) &  (\dots, \bar 4) & 1  \\
f_4 & (\dots, 2,4) &  (\dots, 2,2) & 1 
\end{array}
\hspace{.5in}
\begin{array}{c|r|r|r}
f_j & \multicolumn{1}{c|}{\mathbf b}  & \multicolumn{1}{c|}{f_j({\mathbf b})} & \Delta z \\ \hline
f_1 & (\dots, 2,3^{[m]},2,2)  &  (\dots, 2,3^{[m]}) &  1 \\
f_2 & (\dots, 2,3,2) &  (\dots, 2,2,2) &  1 \\
f_2 & (\dots, 2,3^{[n]},4,2) &  (\dots, 2,3^{[n+1]},2) &  -1 \\
f_3 & (\dots, 2,3^{[m]},2,3) &  (\dots, 2,3^{[m+1]}) &  1 \\
f_4 & (\dots, 2,3^{[m]},4) &  (\dots, 2,3^{[m]},2) &  -1 \\
\multicolumn{4}{c}{}
\end{array}
\renewcommand{\arraystretch}{1}
\]
\end{lemma}

\begin{proof} For $\Delta w$, we use the following partitions of the domains of the functions $f_i$:
\begin{align*}
\text{for }f_1,&~K_{22}(c)=\Gray{K_{222}(c)}\sqcup K_{\bar{3}22}(c),\\
\text{for }f_2, &~K_{\bar{3}2}(c)=K_{232}(c)\sqcup K_{\bar{3}32}(c)\sqcup\Gray{K_{\bar{4}2}(c)},\\
\text{for }f_3,&~K_{3}(c)=\Gray{K_{223}(c)}\sqcup K_{\bar{3}23}(c)\sqcup K_{\bar{3}3}(c),\text{ and}\\
\text{for }f_4,&~K_{\bar{4}}(c) = K_{24}(c)\sqcup \Gray{K_{\bar{3}4}(c)}\sqcup \Gray{K_{\bar{5}}(c)}.
\end{align*}

The subsets shaded gray above are those where $\Delta w=0$, and the subsets shaded black are those which appear in the table above. For each subset one calculates $\Delta$ directly. For example, in the case listed in the second row of the table, if $\mathbf b = (b_1, \dots, b_k,2,3,2)\in K_{232}(c)$, then
\[
w(\mathbf b)  = w((b_1,\dots, b_k)) + 3-\delta_{b_k,2}, 
\]
{where $\delta_{b_k,2} =1 $ if $b_k = 2$ and $\delta_{b_k,2} = 0$ if $b_k \neq 2$ (again using Kronecker's delta notation)}, whereas
\[
w(f_1(\mathbf b)) = w((b_1,\dots, b_k,2,2,2)) = w((b_1,\dots, b_k)) + 1-\delta_{b_k,2}.
\]
Therefore, $\Delta w=2$ in this case. All of the remaining calculations for $\Delta w$ are similar. 

For $\Delta z$, the reasoning is similar, using the following partitions, where the subsets shaded gray are those for which $\Delta z=0$: 
\begin{align*}\pushQED\qed
\text{for $f_1$,}&~K_{22}(c)=\Gray{K_{222}(c)}\sqcup K_{23^{[m\ge 1]}22}(c)\sqcup\Gray{K_{\{ (3^{[m\ge 1]},2,2)\}}(c)} \sqcup \Gray{K_{\bar 4 3^{[n\ge 0]} 22}(c)},\\
\text{for $f_2$,}&~K_{\bar{3}2}(c)=K_{232}(c)\sqcup \Gray{K_{\bar 3 3 2}(c)}\sqcup K_{2 3^{[n\ge 0]}42}(c)\sqcup \Gray{K_{\{(3^{[n\ge 0]},4,2)\}}(c)} \sqcup  \Gray{K_{\bar 4 3^{[n\ge 0]}42}} \sqcup \Gray{K_{\bar 52}(c)},\\
\text{for $f_3$,}&~K_{3}(c)=\Gray{K_{223}(c)}\sqcup K_{23^{[m\ge 1]}23}(c)\sqcup \Gray{K_{\{(3^{[m\ge 1]},2,3)\}}(c)}\sqcup \Gray{K_{\bar 4 3^{[n\ge 0]}23}(c)} \sqcup \Gray{K_{\bar 3 3}(c)},\text{ and}\\
\text{for $f_4$,}&~K_{\bar{4}}(c)= \Gray{K_{24}(c)}\sqcup K_{23^{[m\ge 1]}4}(c)\sqcup \Gray{K_{\{(3^{[m\ge 1]},4)\}}(c)}\sqcup \Gray{K_{\bar 4 3^{[n\ge 0]}4}} \sqcup \Gray{K_{\bar5}(c)}.
\qedhere
\end{align*}

\end{proof}

There is a non-homogeneous term in the recursions for $W(c)$ and $Z(c)$ that will depend on the size of subsets of $K(c)$ where the values of $w$ or $z$ are changed according to Lemma~\ref{L:fwzchange}. The following lemma gives the cardinalities of {these subsets} in terms of the function $j(c)$ which satisfies the homogeneous 
Jacobsthal recursion $j(c) = j(c-1) + 2 j(c-2)$ with initial conditions $j(0)=0$ and $j(1)=1$ \cite{OEIS1045}.  This function is given by 
\begin{equation*}\label{E:j(c)def}
j(c) = \frac{2^c-(-1)^c}{3}
\end{equation*}
and satisfies the identity
\begin{equation}\label{E:j(c)property}
2j(c) - j(c+1) = 2 \left(\frac{2^c-(-1)^c}{3}\right)-\frac{2^{c+1}-(-1)^{c+1}}{3} =(-1)^{c+1}.
\end{equation}
Some small values of $j(c)$ are listed in Table~\ref{T:Jacobsthalnumbers}.

\begin{table}[h]
\renewcommand{\arraystretch}{1.5}
\[
\begin{array}{c|rrrrrrrrrrr}
c & 0 & 1 & 2 & 3 & 4 & 5 & 6 & 7 & 8 & 9 & 10\\\hline
j(c)&  0 & 1 & 1 & 3 & 5 & 11 & 21 & 43 & 85 & 171 & 341
\end{array}
\]
\vspace{2pt}
\caption{Small values in the Jacobsthal recursion.}
\label{T:Jacobsthalnumbers}
\end{table}

\begin{lemma}\label{L:Kcardinalities}
We have the following identities on cardinalities of sets:
\begin{itemize}
\item[(a)] $|K(c)| = 2 j(c-2)$ for $c \ge 2$. \\
\item[(b)] $|K_{22}(c)| = |K_{\bar 4}(c)| = 2j(c-4)$, for $c \ge 4$. \\
\item[(c)] $|K_{2}(c)| = |K_{\bar 3}(c)| =j(c-2)$, for $c \ge 2$.\\
\item[(d)] $|K_{2 3^{[m\ge 1]}22}(c)| = |K_{2 3^{[m\ge 1]}22}(c-2)| + j(c-6) + (-1)^c\cdot \delta_{0,c\text{ mod }3}$, for $c \ge 6$.
\end{itemize}
\end{lemma}

In part (d), we have $\delta_{0,c\text{ mod }3}=1$ if $c\equiv0\Mod{3}$ and otherwise $\delta_{0,c\text{ mod }3}=0$.

\begin{proof}
For part (a), partition $K(c)$ as 
$K(c) = K_{22}(c) \sqcup K_{\bar 32}(c) \sqcup   K_{3}(c) \sqcup  K_{\bar 4}(c)$. Using the bijections $f_1$, $f_2$, $f_3$, and $f_4$ from Lemma~\ref{L:fbijections} we obtain
\[
|K(c)| = |K(c-2)| + |K_2(c-1)| +  |K_{\bar 3}(c-1)| + |K(c-2)|= |K(c-1)| + 2|K(c-2)|.
\]
Solving this homogeneous Jacobsthal recursion subject to the initial conditions $|K(2)| =0$ and $|K(3)| = 2$ gives 
\[
|K(c)| = \frac{2}{3} \, 2^{c-2} - \frac{2}{3} (-1)^{c-2} = 2j(c-2).
\]

For part (b), by Lemma~\ref{L:fbijections}, $f_1:K_{22}(c) \rightarrow K(c-2)$ and $f_4:K_{\bar 4}(c) \rightarrow K(c-2)$ are bijections. Therefore, by part (a) we have
$ |K_{22}(c)| =|K_{\bar 4}(c)| = |K(c-2)| = 2j(c-4)$.

For part (c), we induct on $c \ge 4$. By direct calculation, since $K(4)=\{(2,3),(3,2)\}$, we have $|K_{2}(4)| = 1 = j(1)$.
Now assume $|K_{2}(c-1)| =  j(c-3)$. Then by part (b), Lemma~\ref{L:fbijections},  and the inductive hypothesis we have
\begin{align*}
|K_2(c)| & = |K_{22}(c)| + |K_{\bar 32}(c)| \\
& = 2j(c-4) + |K_2(c-1)| \\
& = 2j(c-4) + j(c-3)\\
& = j(c-2).
\end{align*}
A similar argument applies to $K_{\bar 3}(c)$.

For part (d) we induct on $c \ge 6$. By direct calculation $|K_{23^{[m\ge 1]}22}(4)| = |K_{23^{[m\ge 1]}22}(5)| = 0$ and $|K_{23^{[m\ge 1]}22}(6)| = |K_{23^{[m\ge 1]}22}(7)| = 1$. Hence, $|K_{23^{[m\ge 1]}22}(6)| = 0 + j(0) + (-1)^{6}$ and $|K_{23^{[m\ge 1]}22}(7)| = 0 +  j(1)$.
Thus, the base cases $c=6$ and $c=7$ are given by the formula. Now assume for $6 \le k \le c-1$ that  
\[
|K_{2 3^{[m\ge 1]}22}(k)| = 
|K_{2 3^{[m]\ge 1}22}(k-2)| + j(k-6) + (-1)^{k} \cdot \delta_{0,k~\mbox{\scriptsize mod}~3}.
\]
Consider the partition
\[
K_{2 3^{[m\ge 1]}22}(c) = K_{2322}(c) \sqcup K_{23322}(c) \sqcup K_{233322}(c) \sqcup K_{2 3^{[\ell \ge 4]}22}(c).
\]
If we apply the bijections $f_3 \circ f_3 \circ f_1$, {$f_4 \circ f_3 \circ f_1$}, $f_1 \circ f_2 \circ f_4 \circ f_3 \circ f_1$, and $f_2 \circ f_4 \circ f_3 \circ f_1$ to each of the subsets, respectively, we obtain a bijection
\[
K_{2 3^{[m\ge 1]}22}(c) \rightarrow K_{\bar 3}(c-4) \sqcup K_{22}(c-5) \sqcup K_2(c-8) \sqcup K_{2 3^{[m\ge 1]}22}(c-6).
\]
Therefore, by our inductive hypothesis, part (b), and part (c), it follows that
\begin{align*}
|K_{2 3^{[m\ge 1]}22}(c) | &=  j(c-6) + 2j(c-9)  + j(c-10) + |K_{2 3^{[m\ge 1]}22}(c-6) | \\
&=   j(c-6) + 2j(c-9)  +  |K_{2 3^{[m\ge 1]}22}(c-4)| - (-1)^{c-4} \cdot \delta_{0,c-4~\mbox{\scriptsize mod}~3} \\
&=   j(c-6) + 2j(c-9)   +  |K_{2 3^{[m\ge 1]}22}(c-2)| -  j(c-8) ~- \\
&\phantom{=}~(-1)^{c-2} \cdot \delta_{0,c-2~\mbox{\scriptsize mod}~3} - (-1)^{c-4} \cdot \delta_{0,c-4~\mbox{\scriptsize mod}~3}\\
&=  |K_{2 3^{[m\ge 1]}22}(c-2)|+j(c-6)~+ \\
&\phantom{=}~(-1)^c - (-1)^{c-2} \cdot \delta_{0,c-2~\mbox{\scriptsize mod}~3} - (-1)^{c-4} \cdot \delta_{0,c-4~\mbox{\scriptsize mod}~3},
\end{align*}
where  the last step uses Equation~\eqref{E:j(c)property}. Moreover, note that
\begin{align*}
(-1)^c - (-1)^{c-2} \cdot \delta_{0,c-2~\mbox{\scriptsize mod}~3} - (-1)^{c-4} \cdot \delta_{0,c-4~\mbox{\scriptsize mod}~3} &= 
(-1)^c(1 - \delta_{2,c~\mbox{\scriptsize mod}~3} - \delta_{1,c~\mbox{\scriptsize mod}~3}) \\
&=(-1)^c \cdot \delta_{0,c~\mbox{\scriptsize mod}~3}.
\end{align*}
Therefore, $|K_{2 3^{[m\ge 1]}22}(c)| = |K_{2 3^{[m\ge 1]}22}(c-2)|+j(c-6) + (-1)^c \cdot \delta_{0,c~\mbox{\scriptsize mod}~3}$.

\end{proof}

In order to determine recursive formlulas for $W(c)$ and $Z(c)$, we extend the notation introduced in Equation~\eqref{E:WZ} as follows:
 \[
W_{\mathbf t}(c) = \sum_{\mathbf b\in K_{\mathbf t}(c)}w(\mathbf b),\hspace{1in}Z_{\mathbf t}(c) = \sum_{\mathbf b\in K_{\mathbf t}(c)}z(\mathbf b).\]

\begin{example}\label{Ex:WZ4567}
We have $K(4)=\{(2,3),(3,2)\}$,
so $W_2(4)=2=W_3(4)$, $W(4)=4$, and $Z(4)=0$.  Also, $K(5)=\{(2,2,2,2),(3,2,2),(4,2),(2,2,3),(2,4),(5)\}$, so $W_2(5)=W_{\bar 3}(5)=1+2+2=5$, $W(5)=10$, and $Z(5)=0$. The only tuples in $K(6)$ for which $z$ is nonzero are $(2,3,2,2)$ and $(2,2,3,2)$, so $Z(6)=2$, and the tuples in $K(7)$ for which $z$ is nonzero are $(2,2,3,2,2)$, $(3,2,3,2)$,   $(2,3,3,2)$, and $(2,3,2,3)$, so $Z(7)=4$. 
\end{example}

We are now prepared to determine closed formulas for both $W(c)$ and $Z(c)$.

\begin{prop}\label{P:WZforumlas} We have the following formulas for $W(c)$ and $Z(c)$:
\begin{enumerate}
\item[(a)] $\displaystyle W(c) = c \cdot 2^{c-4}$ for $c \ge 4$ and \\
\item[(b)] $\displaystyle Z(c) = \frac{3c-8}{27} (2^{c -4}) + \frac{14}{27} (-1)^{c} - \frac{2}{3}(-1)^{c} \delta_{1,c\text{ mod }3}
$ for $c \ge 6$.

\end{enumerate}
\end{prop}

\begin{proof}If we partition $K(c)$ into four subsets $K(c) = K_{22}(c) \sqcup K_{\bar 32}(c)  \sqcup  K_{3}(c) \sqcup  K_{\bar  4}(c)$,
then 
\begin{equation*}\label{E:Gammaformula1}
W(c) = W_{22}(c) + W_{\bar 32}(c)  + W_{3}(c) + W_{\bar  4}(c).
\end{equation*}
Lemma~\ref{L:fwzchange} and the first partition described in its proof describe how $w$ changes for each of the subsets in this partition.
For example, 
\[
f_1 : K_{22}(c) = \Gray{K_{222}(c)} \sqcup K_{\bar 322}(c) \rightarrow \Gray{K_2(c-2)} \sqcup K_{\bar 3}(c-2)
= K(c-2)
\]
is a bijection where $\Delta w = 0$ for each vector in $K_{222}(c)$ and $\Delta w = 1$ for each vector in $K_{\bar 322}(c)$. Therefore, 
\[
W_{22}(c) = W_2(c-2) + W_{\bar 3}(c-2) + |K_{\bar 322}(c)| = W(c-2) + |K_{\bar 322}(c)|.
\]
By similar arguments applied to $W_{\bar 32}(c)$,  $W_{3}(c)$, and $W_{\bar  4}(c)$, we have
\begin{align}
\nonumber
W(c) =&
{W(c-2) + |K_{\bar 322}(c)|} + W_2(c-1) + 2|K_{232}(c)|+|K_{\bar 332}(c)| \\ \nonumber
 & + W_{\bar 3}(c-1) + |K_{\bar 323}(c)| + |K_{\bar 33}(c)| + 
{W(c-2) +|K_{24}(c)|} \\
\nonumber
=& \ W(c-1) + 2 W(c-2)  \\
\nonumber
&  + 2|K_{232}(c)| + |K_{\bar 332}(c)| +|K_{\bar 322}(c)| + |K_{\bar 323}(c)| + |K_{\bar 33}(c)|+|K_{24}(c)| . \label{wk}
\end{align}
To simplify the term $D_w(c) =  2|K_{232}(c)| + |K_{\bar 332}(c)| +|K_{\bar 322}(c)|+ |K_{\bar 323}(c)| + |K_{\bar 33}(c)|+|K_{24}(c)|$ on the righthand side of the equation, we use the bijections $f_1$, $f_2$, $f_3$, and $f_4$.
\begin{align*}
D_w(c) &=   2|K_{232}(c)| + |K_{\bar 332}(c)| +|K_{\bar 322}(c)|+ |K_{\bar 323}(c)| + |K_{\bar 33}(c)| + |K_{24}(c)| \\
&=    2|K_{222}(c-1)| + |K_{\bar 322}(c-1)| +|K_{\bar 3}(c-2)|+ |K_{\bar 33}(c-1)| + |K_{\bar4}(c-1)|\\
& ~ + |K_{22}(c-2)| \\
&=   2|K_{2}(c-3)| + |K_{\bar 3}(c-3)| +|K_{\bar 3}(c-2)| + |K_{\bar 4}(c-2)| + |K_{\bar4}(c-1)| + |K_{22}(c-2)|. 
\end{align*}

\noindent
By Lemma~\ref{L:Kcardinalities}, this gives
\[D_w(c) = j(c-4)+3\cdot j(c-5)+4\cdot j(c-6)+2\cdot j(c-5)=3\cdot 2^{c-5}.\]
Hence, we have the recursion $W(c) = W(c-1) + 2 W(c-2) + 3 \cdot 2^{c-5}$. Because the homogeneous recursion $W(c) = W(c-1) + 2 W(c-2)$ has a general solution $u_0(2^{c}) + v_0(-1)^c$, we solve the non-homogeneous recursion $W(c) = W(c-1) + 2 W(c-2) + 3 \cdot 2^{c-5}$ by determining appropriate values of coefficients in a solution of the form $(u_0+u_1 c)(2^{c}) + v_0(-1)^c$.
In Example~\ref{Ex:WZ4567}, we found the initial values $W(4) = 4$ and $W(5)=10$. 
 This gives the closed form $W(c) = c \cdot 2^{c-4}$ for $c \ge 4$.

In a similar manner,
\begin{equation*}\label{E:Gammaformula2}
Z(c) = Z_{22}(c) + Z_{\bar 32}(c)  + Z_{3}(c) + Z_{\bar  4}(c).
\end{equation*}
By Lemmas~\ref{L:fbijections} and \ref{L:fwzchange}, we then have
\begin{align}
\nonumber Z(c) =&\    Z(c-2) + |K_{2 3^{[m\ge 1]}22}(c)| + Z_{2}(c-1) + |K_{232}(c)| - |K_{2 3^{[n\ge 0]}42}(c)| +  \\
\nonumber & +  Z_{\bar 3}(c-1) + |K_{2 3^{[m\ge 1]}23}(c)| +Z(c-2) - |K_{2 3^{[m\ge 1]}4}(c)| \\
\nonumber
=& \ Z(c-1) + 2 Z(c-2)  \\
\nonumber&  + |K_{2 3^{[m\ge 1]}22}(c)| +|K_{232}(c)| - |K_{2 3^{[n\ge 0]}42}(c)|+  |K_{2 3^{[m\ge 1]}23}(c)| - |K_{2 3^{[m\ge 1]}4}(c)|.  \label{zk}
\end{align}
In order to simplify 
\[
D_z(c) = |K_{2 3^{[m\ge 1]}22}(c)| +|K_{232}(c)| - |K_{2 3^{[n\ge 0]}42}(c)|+  |K_{2 3^{[m\ge 1]}23}(c)| - |K_{2 3^{[m\ge 1]}4}(c)|
\] we again use the bijections $f_1$, $f_2$, $f_3$, $f_4$ and Lemma~\ref{L:Kcardinalities}. It will be helpful to  further partition  $K_{2 3^{[n\ge 0]}42}(c)$, $K_{2 3^{[m\ge 1]}23}(c)$, and $K_{23^{[m\ge 1]}4}(c)$ for the bijections below. These bijections will allow us to count the contributions to $D_z(c)$ in terms of the function $j(c)$ and $|K_{2 3^{[m\ge 1]}22}(k)|$.
\begin{align*}
f_1  \circ f_2 &: K_{232}(c) \rightarrow K_2(c-3) \\
f_1 \circ f_2 \circ f_2 &: K_{242}(c) \rightarrow K_2(c-4) \\
f_2 \circ f_2 &: K_{2 3^{[m\ge 1]}42}(c) \rightarrow K_{2 3^{[m\ge 1]}22}(c-2), \\
f_4 \circ f_3 \circ f_3 &: K_{2323}(c) \rightarrow K_{22}(c-4)\\
f_3 \circ f_3 &: K_{23^{[p\ge 2]}23}(c) \rightarrow K_{23^{[m \ge 1]}4}(c-2),\\
f_1 \circ f_2  \circ f_4 &: K_{234}(c) \rightarrow K_2(c-5) \\
f_2  \circ f_4 &: K_{23^{[p\ge 2]}4}(c) \rightarrow K_{23^{[m \ge 1]}22}(c-3).
\end{align*}

We now simplify $D_z(c)$ as follows:
\begin{align*}
D_z(c) =&\ |K_{2 3^{[m\ge 1]}22}(c)| + |K_{232}(c)| - |K_{2 3^{[n\ge 0]}42}(c)| + |K_{2 3^{[m\ge 1]}23}(c)| - |K_{2 3^{[m\ge 1]}4}(c)| \\
=&\ |K_{2 3^{[m\ge 1]}22}(c)| + |K_{232}(c)| -|K_{242}(c)| - |K_{2 3^{[m\ge 1]}42}(c)| + |K_{2323}(c)|+ |K_{2 3^{[p\ge 2]}23}(c)| \\
&- |K_{234}(c)|- |K_{2 3^{[p\ge 2]}4}(c)| \\
=&\ |K_{2 3^{[m\ge 1]}22}(c)| + |K_2(c-3)| -|K_{2}(c-4)| - |K_{2 3^{[m\ge 1]}22}(c-2)| + |K_{22}(c-4)| \\
& + |K_{2 3^{[m\ge 1]}4}(c-2)| - |K_{2}(c-5)|- |K_{2 3^{[m\ge 1]}22}(c-3)| \\
=&\ |K_{2 3^{[m\ge 1]}22}(c)| +j(c-5) -j(c-6) - |K_{2 3^{[m\ge 1]}22}(c-2)| + 2j(c-8)+ |K_{234}(c-2)| \\
& + |K_{2 3^{[p\ge 2]}4}(c-2)|- j(c-7)- |K_{2 3^{[m\ge 1]}22}(c-3)| \\
=&\ |K_{2 3^{[m\ge 1]}22}(c)| +j(c-5) -j(c-6) - |K_{2 3^{[m\ge 1]}22}(c-2)| + 2j(c-8)+ |K_{2}(c-7)| \\
& + |K_{2 3^{[m\ge 1]}22}(c-5)|- j(c-7)- |K_{2 3^{[m\ge 1]}22}(c-3)| \\
=&\ |K_{2 3^{[m\ge 1]}22}(c)| +j(c-5) - \Big[j(c-6) + |K_{2 3^{[m\ge 1]}22}(c-2)| \Big] + 2j(c-8)\\
& + \Big[ j(c-9)+ |K_{2 3^{[m\ge 1]}22}(c-5)| \Big]- j(c-7)- |K_{2 3^{[m\ge 1]}22}(c-3)|.
\end{align*}
Using Lemma~\ref{L:Kcardinalities} to rewrite 
\begin{align*}
j(c-6) + |K_{2 3^{[m\ge 1]}22}(c-2)| &= |K_{2 3^{[m\ge 1]}22}(c)| - (-1)^c \delta_{0,c\text{ mod }3} \ \mbox{and} \\
j(c-9)+ |K_{2 3^{[m\ge 1]}22}(c-5)| &= |K_{2 3^{[m\ge 1]}22}(c-3)|- (-1)^{c-3} \delta_{0,c\text{ mod }3},
\end{align*}
we further simplify:
\begin{align*}
D_z(c) =&    |K_{2 3^{[m\ge 1]}22}(c)| +j(c-5) - \Big[|K_{2 3^{[m\ge 1]}22}(c)| - (-1)^c \delta_{0,c\text{ mod }3}\Big]+ 2j(c-8)\\
&+ \Big[|K_{2 3^{[m\ge 1]}22}(c-3)|- (-1)^{c-3} \delta_{0,c\text{ mod }3}\Big] - j(c-7)- |K_{2 3^{[m\ge 1]}22}(c-3)| \\
=&\ j(c-5) - j(c-7) + 2j(c-8) + 2 (-1)^c \delta_{0,c\text{ mod }3} \\
=&\ 2j(c-6) + 2 (-1)^c \delta_{0,c\text{ mod }3}.
\end{align*}
Therefore:
\[
Z(c) = Z(c-1) +2 Z(c-2) + 2j(c-6) + 2 (-1)^c \delta_{0,c\text{ mod }3}.
\]
As with $W(c)$, the recursion $Z(c) = Z(c-1) + 2 Z(c-2)$ has a general homogeneous solution $u_0(2^{c}) + v_0(-1)^c$. Since the heterogeneous part of the recursion has the form 
\[
\frac{2}{3}(2^{c-6})-\frac{2}{3}(-1)^{c-6}+  2 (-1)^c \delta_{0,c\text{ mod }3},
\]
we guess a particular solution of the form 
\[
u_1 c (2^{c-6})+ v_1 c (-1)^{c-6}+  r_0 (-1)^c \delta_{0,c\text{ mod }3} + s_0 (-1)^c \delta_{1,c\text{ mod }3} +t_0 (-1)^c \delta_{2,c\text{ mod }3}.
\]
Using the recursion and the initial conditions $Z(6)=2$ and $Z(7)=4$ from Example~\ref{Ex:WZ4567}, we solve for the coefficients to obtain the closed form
\[\pushQED\qed
Z(c) = \frac{3c-8}{27} (2^{c -4}) + \frac{14}{27} (-1)^{c} -\frac{2}{3}(-1)^c \delta_{1,c\text{ mod }3}.
\qedhere \]
\end{proof}

\subsection{The palindromic contribution to average unoriented genus}

We now turn our attention to palindromes in order to determine how to adjust the values of $W(c)$ and $Z(c)$ in 
order to reflect the fact that 
$(b_1,\dots,b_k)$ and $(b_k,\dots, b_1)$ in $K(c)$ represent equivalent knots in $\mathcal K(c)$.
Let
\[K^P(c)=\{\mathbf{b}\in K(c):~\mathbf{b}\text{ is palindromic}\}.\]
Furthermore, let
\[K^P_{\mathbf{t}}(c)=\{\mathbf{b}\in K^P(c)\cap K_{\mathbf{t}}(c):~\text{length}(\mathbf{b})\geq 2\, \text{length}(\mathbf{t})-1\}.\]
In other words, $K^P_{\mathbf{t}}(c)$ consists of all palindromic tuples in $K(c)$ that end with $\mathbf t = (b_j, \dots, b_k)$ and, therefore, start with $(b_k, \dots, b_j)$. However, note that these two sub-vectors might overlap at the middle spot, if there is one. For example, $(2,3,3,3,2) \in K^P_{332}(9)$. Given any palindrome $\mathbf b$, we will also write
\[
K^P_{\{ \mathbf b\}}(c) = K^P(c) \cap \{\mathbf b\}
\]
and, as with $K_{\{\mathbf b \}}(c)$, notice that this set might be empty. As in the last section we retain the notation that $b^{[m]}$ represents a string of $m$ $b$'s in $\mathbf t$ and that $b^{[m\ge 1]}$ in $\mathbf t$ represents a union of sets where $m$ takes on all values $m \ge 1$.

\vspace{10pt}
\noindent
The functions $f_1$, $f_2$, $f_3$, and $f_4$ can now be modified to preserve palindromes in $K^P(c)$. In particular, for $c \ge 7$, we define four functions
\begin{align*}
p_1 :&\, K^P_{22}(c) \rightarrow K^P(c-4) \\
p_2 :&\,  K^P_{\bar 32}(c) \rightarrow K^P_{2}(c-2) \\
p_3 :&\, K^P_{3}(c) \rightarrow K^P_{\bar 3}(c-2) \\
p_4 :&\, K^P_{\bar 4}(c) \rightarrow K^P(c-4)
\end{align*}
as follows where $k$ is the length of the vector:
\[
\arraycolsep=1.4pt\def\arraystretch{1}
\begin{array}{lll}
p_1(2,2,b_3,\dots ,b_3,2,2) &  =  (b_3, \dots, b_3), & \mbox{for $k \ge 5$}, \\
p_2(2,b_{2},\dots,b_{2},2) & = \left\{ \begin{array}{l} (2, b_{2}-1, \dots, b_{2}-1,2), \\ (2, b_2-2,2), \end{array} \right. & \begin{array}{l}\mbox{for $k \ge 4$ and $b_2 \ge 3$,} \\ \mbox{for $k = 3$ and $b_2 \ge 4$,} \end{array} \\
p_3(3,b_2,\dots,b_2,3) & =  \left\{ \begin{array}{l} (b_2+1, \dots, b_2+1), \\ (b_1+2), \end{array} \right. &  \begin{array}{l} \mbox{for $k \ge 4$}, \\ \mbox{for $k = 3$,} \end{array} \\
p_4(b_1,b_2,\dots,b_2,b_1) & =  \left\{ \begin{array}{l} (b_1-2, b_2, \dots, b_2,b_1-2), \\ (b_1-4), \end{array} \right. & \begin{array}{l} \mbox{for $k \ge 2$ and $b_1  \ge 4$,} \\ \mbox{for $k = 1$ and $b_1 \ge 7$.} \end{array}.
\end{array}
\]

\noindent Using a figure similar to Figure~\ref{Fi:f1f2f3f4}, we obtain the following lemma.

\begin{lemma}\label{L:pbijections}
The functions $p_1$, $p_2$, $p_3$, and $p_4$ defined above are bijections. 
Hence, for $c \ge 7$, the function $g^P=p_1\cup p_2\cup p_3\cup p_4 : K^P(c) \rightarrow K^P(c-2) \sqcup K^P(c-4) \sqcup K^P(c-4)$
is a bijection. Consequently, $K^P(c) = \varnothing$ when $c$ is even.
\end{lemma}

The final remark in Lemma~\ref{L:pbijections} is a consequence of the recursion induced by the bijection and the initial conditions $K^P(4) = K^P(6) = \varnothing$ (see Example~\ref{Ex:Kp7bijectionexample} below). 

\begin{example}
\label{Ex:Kp7bijectionexample} 
The bijection $g^P:K(7) \rightarrow K(5) \sqcup K(3) \sqcup K(3)$ is shown explicitly in Table~\ref{T:Kp7bijection}. On the left the partition $K^P(7) = K^P_{22}(7) \sqcup K^P_{\bar 32}(7) \sqcup K^P_3(7) \sqcup K^P_{\bar 4}(7)$ is depicted in four boxes. The images $p_1(K^P_{22}(7))=K(3)$ and $p_4(K^P_{\bar 4}(7))=K^P(3)$ are depicted as the topmost and bottommost boxes on the right. The image $p_2 \cup p_3 (K^P_{\bar 32}(7) \sqcup K^P_3(7)) = K^P(5)$ is depicted in the middle box on the right.

\begin{table}[h]
\[
\begin{array}{l|c|c|c|c} \cline{2-2}\cline{4-4}
\multirow{2}{3em}{$K^P_{22}(7)$} & (2,2,3,2,2) & \multirow{2}{2em}{$\stackrel{p_1}{\longrightarrow}$} & (3) & \multirow{2}{4em}{$K^P(3)$} \\
& (2,2,2,2,2,2) & & (2,2) & \\ [2pt] \cline{2-2} \cline{4-4} 
K^P_{\bar 32}(7)\ & (2,3,3,2) & \hspace{-5pt}\stackrel{p_2}{\longrightarrow} & (2,2,2,2) & \multirow{2}{4em}{$K^P(5)$}\\
[2pt]\cline{2-2} \arrayrulecolor{lightgray}
 \cline{4-4} \arrayrulecolor{black}
K^P_{3}(7) & (3,3,3) & \hspace{-5pt}\stackrel{p_3}{\longrightarrow} & (5) & \\ [2pt] \cline{2-2 } \cline{4-4}
\multirow{2}{3em}{$K^P_{\bar 4}(7)$} & (4,4) & \multirow{2}{2em}{$\stackrel{p_4}{\longrightarrow}$} & (2,2) & \multirow{2}{4em}{$K^P(3)$} \\
& (7) & & (3) & \\ \cline{2-2} \cline{4-4}
\end{array}
\]
\caption{The bijection $g^P:K^P(7) \rightarrow K^P(3) \sqcup K^P(5) \sqcup K^P(3)$}
\label{T:Kp7bijection}
\end{table}

\end{example}

\begin{lemma}\label{pxyzchange}
Given $\mathbf b \in \mbox{domain}(p_j)$ for $1 \le j \le 4$, let $\Delta w = w (\mathbf b) - w (p_j(\mathbf b))$ and $\Delta z = z (\mathbf b) - z (p_j(\mathbf b))$. Then $\Delta w$ and $\Delta z$ as given in the tables  below where $n \ge 0$ and $m \ge 1$ are integers. For any vector in the domain of $p_j$ that is not listed in the table we have $\Delta w = 0$ or $\Delta z = 0$.

\[
\renewcommand{\arraystretch}{1.25}
\begin{array}{c|r|r|c}
p_j & \multicolumn{1}{c|}{\mathbf b}  & \multicolumn{1}{c|}{p_j({\mathbf b})} &  \Delta w \\ \hline
p_1 & (\dots, \bar 3,2,2) &  (\dots, \bar 3) & 2  \\
p_2 & (\dots, 2,3,2) &  (\dots, 2,2,2) & 4  \\
p_2 & (2,3,3,2) &  (2,2,2,2) & 3  \\
p_2 & (\dots, \bar 3,3,2) &  (\dots, \bar 3,2,2) & 2  \\
p_3 & (\dots, \bar 3,2,3) &  (\dots, \bar 3,3) & 2  \\
p_3 & (\dots, \bar 3,3) &  (\dots, \bar 4) & 2  \\
p_4 & (\dots, 2,4) &  (\dots, 2,2) & 2 \\
\multicolumn{4}{c}{} \\
\multicolumn{4}{c}{} \\
\multicolumn{4}{c}{} \\
\multicolumn{4}{c}{} \\
\end{array}
\hspace{.35in}
\begin{array}{c|r|r|r}
p_j & \multicolumn{1}{c|}{\mathbf b}  & \multicolumn{1}{c|}{p_j(\mathbf b)} & \Delta z \\ \hline
p_1 & (2, 2,3^{[m]},2,2)  &  (3^{[m]}) &  1 \\
p_1 & (\dots, 2,3^{[m]},2,2)  &  (\dots, 2,3^{[m]}) &  2 \\
p_2 & (2,3,3,2) & (2,2,2,2) & 1 \\
p_2 & (\dots, 2,3,2) &  (\dots, 2,2,2) &  2 \\
p_2 & (2, 4, 3^{[n]},4,2) & (2,3^{[n+2]},2) & -1 \\
p_2 & (\dots, 2,3^{[n]},4,2) &  (\dots, 2,3^{[n+1]},2) &  -2 \\
p_3 & (3, 2,3^{[m]},2,3) &  (3^{[m+2]}) &  1 \\
p_3 & (\dots, 2,3^{[m]},2,3) &  (\dots, 2,3^{[m+1]}) &  2 \\
p_4 & (\dots, 2,3^{[m]},4) &  (\dots, 2,3^{[m]},2) &  -2 \\
p_4 & (4,3^{[m]},4) &  (2,3^{[m]},2) &  -1 \\
\multicolumn{4}{c}{}
\end{array}
\renewcommand{\arraystretch}{1}
\]
\end{lemma}

\begin{proof} The proof is similar to that of Lemma~\ref{L:fwzchange}. Here, we use the following partitions of the domains of the functions $p_i$.  For $\Delta w$, we use these partitions:
\begin{align*}
\text{for }p_1,&~K^P_{22}(c)=\Gray{K^P_{222}(c)}\sqcup K^P_{\bar322}(c),\\
\text{for }p_2,&~K^P_{\bar{3}2}(c)=K^P_{232}(c) \sqcup K^P_{\{(2, 3, 3, 2)\}}(c)\sqcup K^P_{\bar3 32}(c) \sqcup \Gray{K^P_{\bar{4}2}(c)},\\
\text{for }p_3,&~K^P_{3}(c)=\Gray{K^P_{223}(c)}\sqcup K^P_{\bar{3}23}(c)\sqcup K^P_{\bar{3}3}(c),\text{ and}\\
\text{for }p_4,&~K_{\bar{4}}(c)=K_{24}(c)\sqcup \Gray{K_{\bar{3}4}(c)}\sqcup \Gray{K_{\bar{5}}(c)}.
\end{align*}
Like before, the subsets shaded gray are those where $\Delta w=0$, and the subsets shaded black are those which appear in the table above. Similarly, for $\Delta z$, we use these partitions:
\begin{align*}
\text{for }p_1,&~K^P_{22}(c)=K^P_{\{(2,2,3^{[m\ge 1]},2,2)\}}(c) \sqcup{K^P_{2 3^{[m\ge 1]} 2 2}(c)}\sqcup\Gray{K^P_{\bar 4 3^{[n\ge 1]} 2 2}(c)},\\
\text{for }p_2,&~K^P_{\bar{3}2}(c)=K^P_{\{(2,3,3,2)\}}(c) \sqcup K^P_{232}(c)\sqcup\Gray{K_{\bar 3 3 2}(c)}\sqcup K^P_{\{(2,4,3^{[n\ge 0]},4,2)\}}(c) \sqcup K^P_{23^{[n\ge 0]}42}(c)\, \sqcup\\
&~\phantom{K^P_{\bar{3}2}(c)=\ } \Gray{K^P_{\bar 4 3^{[n\ge 0]}4 2}(c)}\sqcup\Gray{K^P_{\bar 5 2}(c)},\\
\text{for }p_3,&~K^P_{3}(c)=\Gray{K^P_{2 2 3}(c)}\sqcup K^P_{\{(3,2,3^{[m \ge 1]},2,3)\}}(c)\sqcup{K^P_{2 3^{[m\ge 1]} 2 3}(c)}\sqcup \Gray{K^P_{\bar 4 3^{[n\ge 0]} 2 3}(c)}\sqcup \Gray{K^P_{\bar 3 3}(c)},\text{ and}\\
\text{for }p_4,&~K^P_{\bar 4}(c)=\Gray{K^P_{2 4}(c)}\sqcup {K^P_{2 3^{[m\ge 1]} 4}(c)}\sqcup K^P_{\{(4, 3^{[m\ge 1]}, 4)\}}(c) \sqcup \Gray{K^P_{\bar 4 3^{[n\ge 0]} 4}(c)}\sqcup \Gray{K^P_{\bar 5}(c)}.\\
\end{align*}
The calculations of $\Delta w$ and $\Delta z$ are left to the reader.
\end{proof}

The following lemma is required to determine $W^P(c)$ and $Z^P(c)$ in Proposition~\ref{P:WpZpformulas}.

\begin{lemma}\label{KpcTpc}
We have the following identities on cardinalities of sets of palindromes for $c$ odd, $d = \frac{c-1}{2}$, and $j(d) = \frac{1}{3}(2^d-(-1)^d)$:
\begin{itemize}
\item[(a)] $\displaystyle |K^P(c)| = 2j(d)$ for $c \ge 3$. \\
\item[(b)] $|K^P_{22}(c)| = |K^P_{\bar 4}(c)| = \displaystyle 2j (d-2)$, for $c \ge 7$.\\
\item[(c)] $|K^P_{2}(c)| = |K^P_{\bar 3}(c)| = \displaystyle j (d)$, for $c \ge 3$. \\
\item[(d)] $|K^P_{\{(2,3^{[d-1]},2)\}}(c)| = |K^P_{\{( 3^{[d]} )\} }(c)| = 1 - \delta_{2, d
\ \mbox{\scriptsize mod}\ 3}$, for $c \ge 3$.
\\
\item[(e)] $|K^P_{2 3^{[m \ge 1]} 22}(c)| =  |K^P_{2 3^{[m \ge 1]} 22}(c-4)| + j(d-4) +\delta(d)$ for $c \ge 17$ and where \\ \\
$\delta(d) = \delta_{1,d~\mbox{\scriptsize mod}~6} + \delta_{2,d~\mbox{\scriptsize mod}~6}-\delta_{3,d~\mbox{\scriptsize mod}~6}-\delta_{5,d~\mbox{\scriptsize mod}~6}$
\end{itemize}
\end{lemma}
In part (e), we have $\delta_{r,d\text{ mod }6}=1$ if $d\equiv r\Mod{6}$ and $\delta_{r,d\text{ mod }6}=0$ otherwise.

\begin{proof}
If $c \ge 7$ is odd, consider the partition 
$
K^P(c) = K^P_{\bar 32}(c) \sqcup K^P_{22}(c) \sqcup K^P_{3}(c) \sqcup K^P_{\bar 4}(c)$.
From Lemma~\ref{L:pbijections} we obtain $|K^P(c)| = |K^P(c-2)|+2|K^P(c-4)|$, and so (a) follows from the initial conditions $|K^P(3)| = |K^P(5)| = 2$ (see Example~\ref{Ex:Kp7bijectionexample}). Formula (b) follows from the bijections $p_1:K^P_{22}(c) \rightarrow K^P(c-4)$ and $p_4:K^P_{\bar 4}(c) \rightarrow K^P(c-4)$ together with part (a). The fomulas in (c) follow by induction. For example, $|K^P_2(3)| = 1 = j(1)$, and from the bijection
\[
p_1 \cup p_2 : K^P_2(c) = K^P_{22}(c) \sqcup K^P_{\bar3 2}(c) \rightarrow  K^P(c-4) \sqcup K^P_2(c-2) 
\]
we have
\[
|K^P_2(c)| = |K^P(c-4)| + |K^P_2(c-2)| =2 j\left( \frac{c-5}{2} \right) + j\left( \frac{c-3}{2}\right)=j\left( \frac{c-1}{2}\right).
\]
The proof for $K^P_{\bar 3}(c)$ is similar.

Now consider the sets $K^P_{\{(2,3^{[d-1]},2)\}}(c)$ and $K^P_{\{(3^{[d]})\}}(c)$. By direct calculation we have 
\[
\begin{array}{c|c|c}
d & K^P_{\{(2,3^{[d-1]},2)\}}(c) & K^P_{\{(3^{[d]})\}}(c) \\ \hline
1 & (2,2) & (3) \\
2 & \varnothing & \varnothing \\
3 & (2,3,3,2) & (3,3,3)
\end{array}
\]
So for these three base cases we have $|K^P_{\{(2,3^{[d-1]},2)\}}(c)|=|K^P_{\{(3^{[d]})\}}(c)|$. Figure \ref{fig:comps} shows that $(2,3^{[d-1]},2)$ and $(2,3^{[d+2]},2)$ are either both knots or both 2-component links; see also Cohen and Krishnan \cite{CK_2015}.  The same is true for $(3^{[d]})$ and $(3^{[d+3]})$. Thus, part (d) is established by induction.

We prove (e) using induction. By direct calculation we have:
\[
\renewcommand{\arraystretch}{1.5}
\begin{array}{c|rrrrrrr}
c & 7 & 9 & 11 & 13 & 15 & 17 & 19 \\ \hline
|K^P_{2 3^{[m\ge 1]}22}(c)| &  0 & 0 & 0 & 1 & 4 & 7 & 14 
\end{array}
\]
and so the recursion is satisfied for odd $c$ with $7 \le c \le 19$. We now induct on $d = \frac{c-1}{2}$ to establish the result for $c \ge 21$. The base cases $d=8$ and $d=9$ are given in above. For the inductive step, consider the partition
\[
K^P_{2 3^{[m\ge 1]}22}(c) = K^P_{2322}(c) \sqcup K^P_{23322}(c) \sqcup K^P_{233322}(c) \sqcup K^P_{2 3^{[\ell \ge 4]}22}(c).
\]
If we apply the bijections $p_3 \circ p_3 \circ p_1$, \, {$p_4 \circ p_3 \circ p_1$}, \, $p_1 \circ p_2 \circ p_4 \circ p_3 \circ p_1$, and $p_2 \circ p_4 \circ p_3 \circ p_1$ to each of the subsets, respectively, we obtain a bijection
\[
K^P_{2 3^{[m\ge 1]}22}(c) \rightarrow K^P_{\bar 3}(c-8) \sqcup K^P_{22}(c-10) \sqcup K^P_2(c-16) \sqcup K^P_{2 3^{[m\ge 1]}22}(c-12).
\]
Therefore, by inductive hypothesis, part (b), and part (c), it follows that
\begin{align*}\pushQED\qed
|K^P_{2 3^{[m\ge 1]}22}(c) | &=  j(d-4) + 2j(d-7)  + j(d-8) + |K^P_{2 3^{[m\ge 1]}22}(c-12) | \\
&=   j(d-4) + 2j(d-7) +  |K^P_{2 3^{[m\ge 1]}22}(c-8)| - \delta(d-4) \\
&=   j(d-4) + 2j(d-7) +  |K^P_{2 3^{[m\ge 1]}22}(c-4)| - j(d-6) - \delta(d-2) - \delta(d-4) \\
&=  |K^P_{2 3^{[m\ge 1]}22}(c-4)| + j(d-4) + (-1)^d - \delta(d-2) - \delta(d-4) \\
&= |K^P_{2 3^{[m\ge 1]}22}(c-4)| + j(d-4) + \delta(d).\qedhere
\end{align*}
\end{proof}

\begin{figure}[h!]
\[\begin{tikzpicture}[scale=.5,thick]

  \draw (0,1) arc (270:90:.4cm and .5cm);
  \draw (0,0) arc (270:90:1.2cm and 1.5cm);
  \draw (0,3) -- (4,3);
  \draw (6,3) -- (9,3);
  \draw (15,3) -- (16,3);
  \draw (15,2) -- (16,2);
  \draw (16,3) arc (90:-90:1.2cm and 1.5cm);
  \draw (16,2) arc (90:-90:.4cm and .5cm);
  
  \begin{knot}[ 	
	consider self intersections,
 	clip width = 3,
 	ignore endpoint intersections = true,
	end tolerance = 2pt
	]
	\flipcrossings{1,3,4,6,9,11,14}
	\strand[looseness=.4] (0,0) to [out = 0, in = 180]
	(2,2) to [out = 0, in = 180]
	(3,2) to [out = 0, in = 180]
	(4,1);
	\strand[looseness=.4] (0,1) to [out = 0, in = 180]
	(1,0) to [out = 0, in = 180]
	(2,0) to [out = 0, in = 180]
	(4,2);
	\strand[looseness=.4] (0,2) to [out = 0, in = 180]
	(1,2) to [out =0, in = 180]
	(3,0) to [out = 0, in = 180]
	(4,0);
	
	\strand[looseness=.4] (6,0) to [out = 0, in = 180]
	(7,0) to [out = 0, in = 180] 
	(9,2);
	\strand[looseness=.4] (6,1) to [out = 0, in = 180]
	(7,2) to [out = 0, in = 180]
	(8,2) to [out = 0, in = 180]
	(9,1);
	\strand[looseness=.4] (6,2) to [out = 0, in = 180]
	(8,0) to [out = 0, in = 180]
	(9,0);
	\strand[looseness=.4] (9,-3) -- (15,-3);
	\strand[blue, looseness=.4] (9,-4) to [out = 0, in = 180]
	(10,-4) to [out = 0, in = 180] 
	(12,-6) to [out = 0, in = 180]
	(13,-6) to [out = 0, in = 180]
	(15,-4);
	\strand[red, looseness=.4] (9,-5) to [out = 0, in = 180]
	(10,-6) to [out = 0, in = 180]
	(11,-6) to [out = 0, in =180]
	(13,-4) to [out = 0, in =180]
	(14,-4) to [out = 0, in = 180]
	(15,-5);
	\strand[black!20!green, looseness=.4] (9,-6) to [out = 0, in = 180]
	(11,-4) to [out = 0, in = 180]
	(12,-4) to [out = 0, in= 180]
	(14,-6) to [out = 0, in =180]
	(15,-6);
	\strand[looseness=.4] (15,0) to [out = 0, in =180]
	(16,1);
	\strand[looseness=.4] (15,1) to [out = 0, in =180]
	(16,0);
	
\end{knot}
	
	\fill (5,1.5) circle (.05cm);
	\fill (4.8,1.5) circle (.05cm);
	\fill (5.2,1.5) circle (.05cm);
	\draw (9,3) -- (15,3);
	\draw[blue] (9,2) -- (15,2);
	\draw[red] (9,1) -- (15,1);
	\draw[black!20!green] (9,0) -- (15,0);
	\draw[thin] (15,-.5) -- (15,3.5);
	\draw[thin] (9,-.5) -- (9,3.5);
	\draw[thin] (15,-2.5) -- (15,-6.5);
	\draw[thin] (9,-2.5) -- (9,-6.5);
	\draw[thick, ->] (12,-2) -- (12,-1);
	
	\draw (.5,1.5) node{$2$};
	\draw (2.5,1.5) node{$3$};
	\draw (7.5,1.5) node{$3$};
	\draw (15.5,1.5) node{$2$};
	
	\draw (9.5,-4.5) node{$3$};
	\draw (11.5,-4.5) node{$3$};
	\draw (13.5,-4.5) node{$3$};
  
\end{tikzpicture}\]
\caption{Inserting the depicted tangle preserves the number of components.}
\label{fig:comps}
\end{figure}

In the following proposition, we give closed forms for $W^P(c)$ and $Z^p(c)$.
\begin{prop}
\label{P:WpZpformulas}
We have the following formulas for $W^P(c)$ and $Z^P(c)$. For $c$ even, $W^P(c)=0=Z^P(c)$. 
For $c$ odd and $c \ge 11$,  writing $d=\frac{c-1}{2}$, we have the following
\begin{enumerate}
    \item[(a)] $\displaystyle W^P(c) = \frac{1+3d}{3} \left( 2^{d-1} \right) - \frac{2}{3}(-1)^{d}$, \\
    \item[(b)] $\displaystyle Z^P(c) =  \frac{(3d+1)}{27}(2^{d-1})- \frac{14}{27}(-1)^{d}+\frac{2}{3}(-1)^{d}\left(\delta_{1,d~\text{mod}~3}+3\delta_{2,d~\text{mod}~3}\right)$.
\end{enumerate}
\end{prop}

\begin{proof} We proceed as in Proposition~\ref{P:WZforumlas}. We assume $c$ odd and $c \ge 11$. By starting the recursion at $c=11$ we avoid contributions to $\Delta w$ from $\mathbf b = (2,3,3,2)$ and can also use Lemma~\ref{KpcTpc} to find $|K^P_{22}(c-4)|$.
If we partition $K^P(c)$ into four subsets $K^P(c) = K^P_{22}(c) \sqcup K^P_{\bar 32}(c)  \sqcup  K^P_{3}(c) \sqcup  K^P_{\bar  4}(c)$,
then by Lemma~\ref{pxyzchange} we have
\begin{align*}
W^P(c) =&~W^P_{22}(c) + W^P_{\bar 32}(c)  + W^P_{3}(c) + W^P_{\bar  4}(c) \\
=&~W^P(c-4) + 2|K^P_{\bar322}(c)| + W^P_2(c-2) + 4|K^P_{232}(c)|+ 2| K^P_{\bar3 32}(c)| +  \\
&~W^P_{\bar 3}(c-2) + 2|K^P_{\bar{3}23}(c)|+ 2|K^P_{\bar{3}3}(c)| + W^P(c-4) + 2|K^P_{24}(c)| \\
=&~W^P(c-2) + 2 W^P(c-4) + D^P_w(c)
\end{align*}
where
\[
D^P_w(c) = 2\left(|K^P_{\bar322}(c)| +  2|K^P_{232}(c)|+ | K^P_{\bar3 32}(c)| + |K^P_{\bar{3}23}(c)|+ |K^P_{\bar{3}3}(c)| + |K^P_{24}(c)|\right).
\]
We simplify $D^P_w(c)$ using the bijections $p_1$, $p_2$, $p_3$, $p_4$, and Lemma~\ref{KpcTpc}.
\begin{align*}
\frac{D^P_w(c)}{2}  =&~|K^P_{\bar 322}(c)| + 2|K^P_{232}(c)| + |K^P_{\bar 332}(c)| + |K^P_{\bar 323}(c)| + |K^P_{\bar 33}(c)| + |K^P_{24}(c)| \\
=&~|K^P_{\bar 3}(c-4)| + 2|K^P_{222}(c-2)| + |K^P_{\bar 322}(c-2)| + |K^P_{\bar 33}(c-2)| + |K^P_{\bar4}(c-2)| + \\
&~|K^P_{22}(c-4)| \\
=&~j(d-2) + 2|K^P_{2}(c-6)| + |K^P_{\bar 3}(c-6)| + |K^P_{\bar 4}(c-4)| + 2 j(d-3) + 2j(d-4) \\
=&~j(d-2) + 2j(d-3) + j(d-3) + 2j(d-4) + 2 j(d-3) + 2j(d-4) \\
=&~j(d-1) + j(d-2) + j(d-3) + j(d-3)+2j(d-4) \\
=&~j(d-1) + 2j(d-2) + j(d-3) \\
=&~j(d) + j(d-3) \\
=&~3 \cdot 2^{d-3}.
\end{align*}
Therefore, $W^P(c)$ satisfies the recursion
\[
W^P(c) = W^P(c-2)+2W^P(c-4) + 3 \cdot 2^{(c-5)/2}.
\]
Solving the recursion with initial conditions $W^P(7)=14$ and $W^P(9)=34$ gives the closed formula
\[
W^P(c) = \frac{2+6d}{3} \left( 2^{d-2} \right) - \frac{2}{3}(-1)^{d}.
\]

In a similar manner,
\begin{align*}
Z^P(c) =&~Z^P_{22}(c) + Z^P_{\bar 32}(c)  + Z^P_{3}(c) + Z^P_{\bar  4}(c) \\
=&~Z^P(c-4) + |K^P_{\{(2,2,3^{[m\ge 1]},2,2)\}}(c)| + 2|K^P_{2 3^{[m\ge 1]} 2 2}(c)| \, + \\
&~Z^P_2(c-2) +  2|K^P_{232}(c)|- |K^P_{\{(2,4,3^{[n\ge 0]},4,2)\}}(c)| -2 |K^P_{23^{[n\ge 0]}42}(c)|\,  +  \\
&~Z^P_{\bar 3}(c-2) + |K^P_{\{(3,2,3^{[m \ge 1]},2,3)\}}(c)| + 2|{K^P_{2 3^{[m\ge 1]} 2 3}(c)}| \, + \\
&~Z^P(c-4) - 2 |K^P_{2 3^{[m\ge 1]} 4}(c)| - |K_{\{(4, 3^{[m\ge 1]}, 4)\}}(c)| \\
=&~Z^P(c-2) + 2 Z^P(c-4) + D^P_z(c)
\end{align*}
where
\begin{align*}
D^P_z(c) =&~|K^P_{\{(2,2,3^{[m\ge 1]},2,2)\}}(c)| + 2|K^P_{2 3^{[m\ge 1]} 2 2}(c)|  + 2|K^P_{232}(c)| \, -\\
&~|K^P_{\{(2,4,3^{[n\ge 0]},4,2)\}}(c)| -2 |K^P_{23^{[n\ge 0]}42}(c)|+ |K^P_{\{(3,2,3^{[m \ge 1]},2,3)\}}(c)| \, +\\
&~ 2|{K^P_{2 3^{[m\ge 1]} 2 3}(c)}| - 2 |K^P_{2 3^{[m\ge 1]} 4}(c)|  - |K^P_{\{(4, 3^{[m\ge 1]}, 4)\}}(c)|.
\end{align*}
First note that, for $c$ odd and $c \ge 11$, if $d = \frac{c-1}{2}$, then
\[
\begin{array}{ccc}
K^P_{\{(2,2,3^{[m\ge 1]},2,2)\}}(c) = K^P_{\{(2,2,3^{[d-2]},2,2)\}}(c) & & K^P_{\{(2,4,3^{[n\ge 0]},4,2)\}}(c) = K^P_{\{(2,4,3^{[d-4]},4,2)\}}(c) \\
\\
K^P_{\{(3,2,3^{[m \ge 1]},2,3)\}}(c) = K^P_{\{(3,2,3^{[d-3]},2,3)\}}(c) & &  K^P_{\{(4, 3^{[m\ge 1]}, 4)\}}(c) = K^P_{\{(4, 3^{[d-3]}, 4)\}}(c)
\end{array}
\]
Using the bijections, $p_1$ and  $p_4$ together with Lemma~\ref{KpcTpc} we then have
\[
|K^P_{\{(2,2,3^{[d-2]},2,2)\}}(c)| - |K^P_{\{(4, 3^{[d-3]}, 4)\}}(c)| = |K^P_{\{(3^{[d-2]})\}}(c-4)| - |K^P_{\{(2, 3^{[d-3]}, 2)\}}(c-4)| = 0
\]
Similarly, using the bijections $p_2$ and $p_3$ 
\[
-|K^P_{\{(2,4,3^{[d-4]},4,2)\}}(c)| + |K^P_{\{(3,2, 3^{[d-3]}, 2,3)\}}(c)| = -|K^P_{\{(2,3^{[d-2]},2)\}}(c-2)| + |K^P_{\{(3^{[d-1]})\}}(c-2)| = 0
\]
Thus,
\[
D^P_z(c) = 2 \left(|K^P_{2 3^{[m\ge 1]} 2 2}(c)| + |K^P_{232}(c)|   - |K^P_{23^{[n\ge 0]}42}(c)|+ |{K^P_{2 3^{[m\ge 1]} 2 3}(c)}| -  |K^P_{2 3^{[m\ge 1]} 4}(c)|\right) .
\]

In order to simplify 
$D^P_z(c)$ we further partition  $K_{2 3^{[n\ge 0]}42}(c)$, $K_{2 3^{[m\ge 1]}23}(c)$, and $K_{23^{[m\ge 1]}4}(c)$ for the bijections below.
\begin{align*}
p_1  \circ p_2 &: K^P_{232}(c) \rightarrow K^P_2(c-6) \\
p_1 \circ p_2 \circ p_2 &: K^P_{242}(c) \rightarrow K^P_2(c-8) \\
p_2 \circ p_2 &: K^P_{2 3^{[m\ge 1]}42}(c) \rightarrow K^P_{2 3^{[m\ge 1]}22}(c-4), \\
p_4 \circ p_3 \circ p_3 &: K^P_{2323}(c) \rightarrow K^P_{22}(c-8)\\
p_3 \circ p_3 &: K^P_{23^{[\ell\ge 2]}23}(c) \rightarrow K^P_{23^{[m \ge 1]}4}(c-4),\\
p_1 \circ p_2  \circ p_4 &: K^P_{234}(c) \rightarrow K^P_2(c-10) \\
p_2  \circ p_4 &: K_{23^{[\ell \ge 2]}4}(c) \rightarrow K^P_{23^{[m \ge 1]}22}(c-6).
\end{align*}

We now simplify $D^P_z(c)$ using the bijections $p_1$, $p_2$, $p_3$, $p_4$, and Lemma~\ref{KpcTpc}.
\begin{align*}
\frac{D^P_z(c)}{2}  =&~|K^P_{2 3^{[m\ge 1]} 2 2}(c)| + |K^P_{232}(c)|   - |K^P_{23^{[n\ge 0]}42}(c)|+ |{K^P_{2 3^{[m\ge 1]} 2 3}(c)}| -  |K^P_{2 3^{[m\ge 1]} 4}(c)| \\
=&~|K^P_{2 3^{[m\ge 1]} 2 2}(c)| + |K^P_{2}(c-6)| - |K^P_{242}(c)|  - |K^P_{23^{[m\ge 1]}42}(c)|~+ \\
&~~|{K^P_{2323}(c)}| + |{K^P_{2 3^{[\ell \ge 2]} 2 3}(c)}| - |K^P_{2 34}(c)| -   |K^P_{2 3^{[\ell\ge 2]} 4}(c)| \\
=&~|K^P_{2 3^{[m\ge 1]} 2 2}(c)| + |K^P_{2}(c-6)| - |K^P_{2}(c-8)|  - |K^P_{23^{[m\ge 1]}22}(c-4)|~+ \\
&~~|{K^P_{22}(c-8)}| + |{K^P_{2 3^{[m \ge 1]}4}(c-4)}| - |K^P_{2}(c-10)| -   |K^P_{2 3^{[m\ge 1]} 22}(c-6)| \\
=&~|K^P_{2 3^{[m\ge 1]} 2 2}(c)| + j(d-3) - j(d-4)  - |K^P_{23^{[m\ge 1]}22}(c-4)|+2j(d-6)~+ \\
&~~|{K^P_{234}(c-4)}|+ |{K^P_{2 3^{[\ell \ge 2]}4}(c-4)}| - j(d-5) -   |K^P_{2 3^{[m\ge 1]} 22}(c-6)| \\
=&~\delta(d) + j(d-3)+2j(d-6)+ |{K^P_{2}(c-14)}|+ |{K^P_{2 3^{[m \ge 1]}22}(c-10)}|~- \\
&~~  j(d-5) -   |K^P_{2 3^{[m\ge 1]} 22}(c-6)| \\
=&~\delta(d) + j(d-3)+2j(d-6)+ |{K^P_{2 3^{[m \ge 1]}22}(c-10)}|- j(d-5) ~- \\
&~~|{K^P_{2 3^{[m \ge 1]}22}(c-10)}| - \delta(d-3)\\
=&~j(d-3)+ (-1)^{d-5}+\delta(d) - \delta(d-3) \\
=&~j(d-3) + (-1)^d (\delta_{2,d~\mbox{\scriptsize mod}~3} - 2  \delta_{1,d~\mbox{\scriptsize mod}~3}).
\end{align*}

Therefore, $Z^P(c)$ satisfies the recursion
\[
Z^P(c) = Z^P(c-2)+2Z^P(c-4) + 2j(d-3) + 2(-1)^d (\delta_{2,d~\mbox{\scriptsize mod}~3} - 2  \delta_{1,d~\mbox{\scriptsize mod}~3}).
\]
Solving the recursion with initial conditions $Z^P(7)=2$ and $Z^P(9)=4$ gives the closed formula
\[\pushQED{\qed}
Z^P(c) =  \frac{(3d+1)2^{d-1}-14(-1)^{d}}{27}+\frac{2}{3}(-1)^{d}\left(\delta_{1,d~\text{mod}~3}+3\delta_{2,d~\text{mod}~3}\right).\qedhere
\]
\end{proof}

\begin{proof}[Proof of Theorem \ref{T:Ep1}]
Theorem \ref{thm:cc_formula} implies that
\[ \overline{\Gamma}(c) = \frac{1}{|\mathcal K_c|}\left(\sum_{K \in \mathcal K_c} (w(K)-z(K)) \right) = \frac{W(c)-Z(c)+W^p(c)-Z^p(c)}{2|\mathcal{K}_c|}.\] 
Theorem \ref{T:Ernst-Sumners} gives a formula for $|\mathcal{K}_c|$, Proposition \ref{P:WZforumlas} gives formulas for $W(c)$ and $Z(c)$, and Proposition \ref{P:WpZpformulas} gives formulas for $W^p(c)$ and $Z^p(c)$.
\end{proof}

Theorem \ref{T:Ep1} implies the first part of Theorem \ref{T:ccave}.

\section{From unoriented genus to crosscap number} 
\label{sec:epsilon}

{Theorem~\ref{T:2BridgeCCNew} implies that} the following set of tuples describes 2-bridge knots and links for which unoriented genus is strictly less than crosscap number:
\[E(c)=\left\{(e_1,\hdots,e_k)\in(2\mathbb{Z})^k:~|e_i|\geq 4,~c=\sum_{i=1}^k|e_i|-\sum_{i=1}^{k-1}\delta_{\text{sign}(e_i),\text{sign}(e_{i+1})}\right\}.\]
Partition $E(c)=K^E(c)\sqcup L^E(c)$, where the tuples in $K^E(c)$ describe knots and those in $ L^E(c)$ describe links:
\begin{equation*}
\begin{array}{lcl}
\displaystyle K^E(c)=\left\{(e_1,\hdots,e_k)\in E(c):~k\text{ even}\right\}, & & \displaystyle L^E(c)=\left\{(e_1,\hdots,e_k)\in E(c):~k\text{ odd}\right\}.
\end{array}
\end{equation*}
Further partition $K^E(c)=K^E_4(c)\sqcup K^E_{\bar 6}(c)$ and $ L^E(c)=L^E_4(c)\sqcup L^E_{\bar 6}(c)$ by defining
\begin{equation*}
\begin{array}{lcl}
\displaystyle K^E_4(c)=\left\{(e_1,\hdots,e_k)\in K^E(c):~|e_k|=4\right\}, & & \displaystyle K^E_{\bar{6}}(c)=\left\{(e_1,\hdots,e_k)\in K^E(c):~|e_k|\geq6\right\},\\
\displaystyle L^E_4(c)=\left\{(e_1,\hdots,e_k)\in L^E(c):~|e_k|=4\right\}, & & L^E_{\bar{6}}(c)=\left\{(e_1,\hdots,e_k)\in L^E(c):~|e_k|\geq6\right\}.
\end{array}
\end{equation*}
Also, as a technical convenience, set $E(0)=\{()\}=K^E(0)$. The following functions 
\begin{align*}
f_K&:K^E_{\bar{6}}(c)\to K^E(c-2),~f_L:L^E_{\bar{6}}(c)\to L^E(c-2),\\ g_K&:K^E_4(c)\to L^E(c-3)\sqcup L^E(c-4),\text{ and }g_L:L^E_4(c)\to K^E(c-3)\sqcup K^E(c-4) \end{align*} are bijections for $c\geq 4$: 
\begin{align*}
f_K,f_L&:(e_1,\hdots,e_k)\mapsto\begin{cases}
(e_1,\hdots,e_k-2)&e_k>4\\
(e_1,\hdots,e_k+2)&e_k<-4;\end{cases}\\
g_K,g_L&:(e_1,\hdots,e_k)\mapsto (e_1,\hdots,e_{k-1}).
\end{align*}
It follows that
\begin{align*}
|E(c)|&=|K^E_{\bar{6}}(c)|+|K^E_4(c)|+|L^E_{\bar{6}}(c)|+|L^E_4(c)|\\
&=|K^E(c-2)|+|L^E(c-3)|+|L^E(c-4)|+|L^E(c-2)|+|K^E(c-3)|+|K^E(c-4)|\\
&=|E(c-2)|+|E(c-3)|+|E(c-4)|.
\end{align*}
Further, denoting $\Delta(c)=|K^E(c)|-|L^E(c)|$, we also have
\begin{align*}
\Delta(c)&=|K^E_{\bar{6}}(c)|+|K^E_4(c)|-|L^E_{\bar{6}}(c)|-|L^E_4(c)|\\
&=|K^E(c-2)|+|L^E(c-3)|+|L^E(c-4)|-|L^E(c-2)|-|K^E(c-3)|-|K^E(c-4)|\\
&=\Delta({c-2})-\Delta({c-3})-\Delta({c-4}).
\end{align*}
This gives two homogeneous recurrence relations, whose characteristic polynomials are
\begin{equation}\label{E:1}
x^4-x^2-x-1=(x+1)(x^3-x^2-1)
\end{equation}
and
\begin{equation}\label{E:2}
x^4-x^2+x+1=(x+1)(x^3-x^2+1).
\end{equation}
We will express the roots of Equations~\eqref{E:1} and \eqref{E:2} in terms of  
\begin{equation*}\label{E:3}
\begin{array}{lclcl}
\alpha=\root 3 \of{\frac{1}{2}\left(29+3\sqrt{93}\right)}, & &
\omega=e^{\pi i/3}, & & \beta=-\root 3 \of{\frac{1}{2}\left(25+3\sqrt{69}\right)}.
\end{array}
\end{equation*}
Equation~\eqref{E:1} has two real roots, 
\begin{equation}\label{E:xi1}
\begin{array}{lcl}
x_1=-1,& &x_2=\frac{1}{3}\left(1+\alpha+\alpha^{-1}\right)\approx 1.466,
\end{array}
\end{equation}
and two complex conjugate roots $x_3,x_4$ given by 
\begin{equation}\label{E:xi2}
x_{(7\pm 1)/2}=\frac{1}{3}(1-\alpha\omega^{\mp1}-\alpha^{-1}\omega^{\pm1})\approx-0.233\pm 0.793i.
\end{equation}
Equation~\eqref{E:2} has two real roots, 
\begin{equation}\label{E:yi1}
\begin{array}{lcl}
y_1=-1,& &y_2=\frac{1}{3}\left(1+\beta+\beta^{-1}\right)\approx -0.755,
\end{array}
\end{equation}
and two complex conjugate roots $y_3,y_4$ given by 
\begin{equation}\label{E:yi2}
y_{(7\pm 1)/2}=\frac{1}{3}-\beta\omega^{\pm1}-\beta^{-1}\omega^{\mp1}\approx0.877\pm0.745i.
\end{equation}
Then $|E(c)|$ and $\Delta(c)$ are given by functions of the form
\begin{equation*}\label{E:EForm}
|E(c)|=u_1{x_1}^{c-4}+u_2{x_2}^{c-4}+u_3{x_3}^{c-4}+u_4{x_4}^{c-4}
\end{equation*}
and
\[\Delta(c)=v_1{y_1}^{c-4}+v_2{y_2}^{c-4}+v_3{y_3}^{c-4}+v_4{y_4}^{c-4}\]
for some coefficients $u_1,u_2,u_3,u_4,v_1,v_2,v_3,v_4\in\mathbb{C}$, whose values we obtain by using the initial values in Table~\ref{T:EDelta} and Cramer's rule:

\begin{table}
\begin{center}
\begin{tabular}{||c|cc|cc||}
\hline
$c$&$K^E(c)$&$L^E(c)$&$|E(c)|$&$\Delta(c)$\\ 
\hline
4&$\varnothing$&$\{(4),(-4)\}$&2&-2\\
5&$\varnothing$&$\varnothing$&0&0\\
6&$\varnothing$&$\{(6),(-6)\}$&2&-2\\
7&$\{(4,4),(-4,-4)\}$&$\varnothing$&2&2\\
\hline
\end{tabular}
\caption{Knots and links with unoriented genus strictly less than crosscap number}\label{T:EDelta}
\end{center}
\end{table}

\begin{equation}\label{E:ci1}
\begin{array}{lcl}
u_1=\frac{\left|\begin{smallmatrix}2&1&1&1\\
0&x_2&x_3&x_4\\
2&{x_2}^2&{x_3}^2&{x_4}^2\\
2&{x_2}^3&{x_3}^3&{x_4}^3\\
\end{smallmatrix}\right|}{\left|\begin{smallmatrix}1&1&1&1\\
x_1&x_2&x_3&x_4\\
{x_1}^2&{x_2}^2&{x_3}^2&{x_4}^2\\
{x_1}^3&{x_2}^3&{x_3}^3&{x_4}^3\\
\end{smallmatrix}\right|}=\frac{2}{3}, & &
v_1=\frac{\left|\begin{smallmatrix}-2&1&1&1\\
0&y_2&y_3&y_4\\
-2&{y_2}^2&{y_3}^2&{y_4}^2\\
2&{y_2}^3&{y_3}^3&{y_4}^3\\
\end{smallmatrix}\right|}{\left|\begin{smallmatrix}1&1&1&1\\
y_1&y_2&y_3&y_4\\
{y_1}^2&{y_2}^2&{y_3}^2&{y_4}^2\\
{y_1}^3&{y_2}^3&{y_3}^3&{y_4}^3\\
\end{smallmatrix}\right|}=-2,
\end{array}
\end{equation}
and similarly, using bars to denote complex conjugates,
\begin{equation}\label{E:ci2}
u_2\approx 0.727,~
u_3\approx 0.303-0.163i,~
u_4=\overline{u_3},~
v_2\approx 1.090,~
v_3\approx -0.545-0.148i,~ 
v_4=\overline{v_3}.
\end{equation}
Thus, with the values from Equations~\eqref{E:xi1} through \eqref{E:ci2}, we have the following for $c\geq 4$:
\begin{equation}\label{E:K^E(c)}
|K^E(c)|=\frac{1}{2}\left(|E(c)|+\Delta(c)\right)=\frac{1}{2}\sum_{i=1}^4\left(u_i{x_i}^{c-4}+v_iy_i^{c-4}\right)
\end{equation}

Accounting for palindromes in this setting is similar, but a bit more straightforward, as our bijections will allow us to consider only knots, and not 2-component links as well.   Let 
$K^{EP}(c)$ 
be the set of palindromes in $K^E(c)$ (with $K^{EP}(0)=\{()\}=K^E(0)$), 
and partition 
$K^{EP}(c)=K^{EP}_4(c)\sqcup K^{EP}_{\bar{6}}(c)$ 
as above for all $c>0$. The following functions $h:K^{EP}_{4}(c)\to K^{EP}(c-6)\sqcup K^{EP}(c-8)$ and $q:K^{EP}_{\bar 6}(c)\to K^{EP}(c-4)$ are bijections for $c\geq 8$:
\begin{align*}
h:(e_1,e_2,\hdots,e_{k-1},e_k)&\mapsto (e_2,\hdots,e_{k-1}),\\
q:(e_1,\hdots,e_k)&\mapsto \begin{cases}
(e_1-2,e_2,\hdots,e_{k-1},e_k-2)&e_k>4\\
(e_1+2,e_2,\hdots,e_{k-1},e_k+2)&e_k<-4.\end{cases}
\end{align*}
From this it follows that for all $c\geq 8$ we have
\begin{equation*}\label{E:PRecur}
|K^{EP}(c)|=|K^{EP}(c-4)|+|K^{EP}(c-6)|+|K^{EP}(c-8)|.
\end{equation*}
\begin{table}
\begin{center}
\begin{tabular}{|ccc|}
\hline
$c$&$K^{EP}(c)$&$|K^{EP}(c)|$\\
\hline
7&$\{(4,4),(-4,-4)\}$&2\\
8&$\varnothing$&0\\
9&$\varnothing$&0\\
10&$\varnothing$&0\\
11&$\{(6,6),(-6,-6)\}$&2\\
12&$\varnothing$&0\\
13&$\{(4,4,4,4),(-4,-4,-4,-4)\}$&2\\
14&$\varnothing$&0\\
\hline
\end{tabular}
\caption{Palindromic knots with unoriented genus strictly less than crossing number}\label{T:P}
\end{center}
\end{table}

The initial values in Table~\ref{T:P} immediately reveal that $|K^{EP}(c)|=0$ when $c$ is even, and when $c$ is odd these values mimic those in the earlier computation of $u_1,u_2,u_3,u_4$, and so we have the following for all $c\geq 7$:
\begin{equation}\label{E:KEP}
|K^{EP}(c)|=
\frac{1}{2}\delta_{1,c~\mbox{\scriptsize mod}~2}\sum_{i=1}^4u_ix_i^{\frac{c+1}{2}-4}.
\end{equation}
\begin{theorem}\label{T:Epsilon_2}
For $c\geq 7$, the number of distinct $c$-crossing 2-bridge knots (counting non-isotopic mirror images separately) with unoriented genus strictly less than crosscap number is
\begin{equation}\label{E:Epsilon_2}
\frac{1}{2}\left(|K^E(c)|+|K^{EP}(c)|\right)=\frac{1}{4}\sum_{i=1}^4\left(u_ix_i^{\frac{c-7}{2}}\left({x_i}^{\frac{c-1}{2}}+\delta_{1,c~\mbox{\scriptsize mod}~2}\right)+v_iy_i^{c-4}\right),
\end{equation}
where the values of $x_i$, $y_i$, $u_i$, and $v_i$ are given in Equations~\eqref{E:xi1}--\eqref{E:ci2}.\end{theorem}
\begin{proof}
Equations~\eqref{E:K^E(c)} and \eqref{E:KEP} give:
\begin{align*}\pushQED\qed
\frac{1}{2}\left(|K^E(c)|+|K^{EP}(c)|\right)&=\frac{1}{4}\left(\sum_{i=1}^4u_i{x_i}^{c-4}+v_iy_i^{c-4}+\delta_{1,c~\mbox{\scriptsize mod}~2}\sum_{i=1}^4u_i x_i^{\frac{c+1}{2}-4}\right)\\
&=\frac{1}{4}\sum_{i=1}^4\left(u_ix_i^{\frac{c-7}{2}}\left({x_i}^{\frac{c-1}{2}}+\delta_{1,c~\mbox{\scriptsize mod}~2}\right)+v_iy_i^{c-4}\right).\qedhere
\end{align*}
\end{proof}

\begin{theorem}\label{T:Ep2}
For all $c\geq 4$, let $d=\frac{c-1}{2}$, and let the values of $x_i$, $y_i$, $u_i$, and $v_i$ be from Equations~\eqref{E:xi1}--\eqref{E:ci2}. The portion $\varepsilon_2(c)=\frac{\frac{1}{2}\left(|K^E(c)|+|K^{EP}(c)|\right)}{|\mathcal K_c|}$ of 2-bridge knots with unoriented genus strictly less than crosscap number is
\begin{equation}\label{E:Ep2}
\varepsilon_2(c)=  \left\{ \begin{array}{ll}
\displaystyle \sum_{i=1}^4\frac{3\left(u_ix_i^{d-3}\left({x_i}^{d}+\delta_{1,c~\mbox{\scriptsize mod}~2}\right)+v_iy_i^{c-4}\right)}{4(2^{c-2}-1)} & \mbox{$c$ even} \\
\\
\displaystyle \sum_{i=1}^4\frac{3\left(u_ix_i^{d-3}\left({x_i}^{d}+\delta_{1,c~\mbox{\scriptsize mod}~2}\right)+v_iy_i^{c-4}\right)}{4(2^{c-2}+2^{d})} & \mbox{$c \equiv 1 \Mod{4}$} \\
\\
\displaystyle \sum_{i=1}^4\frac{3\left(u_ix_i^{d-3}\left({x_i}^{d}+\delta_{1,c~\mbox{\scriptsize mod}~2}\right)+v_iy_i^{c-4}\right)}{4(2^{c-2}+2^{d}+2)} & \mbox{ $c \equiv 3 \Mod{4}$.}
\end{array}
\right.
\end{equation}
In particular, $\varepsilon_2(c)$ approaches 0 as $c\to\infty$.
\end{theorem}

\begin{proof}
Equation~\eqref{E:Ep2} simply combines Equations~\eqref{E:Ernst-Sumners} and \eqref{E:Epsilon_2}.  
Observe in Equation~\eqref{E:Ernst-Sumners} that for all $c\geq 5$ we have $|\mathcal K_c|\geq\frac{2^{c-2}-1}{3}\geq 2^{c-4}$.  On the other hand, the maximum norm among all the $x_i$ and $y_i$ is $x_2\approx 1.466$, and the maximum norm among all the $u_i$ and $v_i$ is two, so for all $c\geq 7$, Equation~\eqref{E:Epsilon_2} gives $
\frac{1}{2}\left(|K^E(c)|+|K^{EP}(c)|\right)<\frac{1}{4}\cdot 4\cdot(2\cdot x_2^{c-4}\cdot 2+2\cdot x_2^{c-4})=6{x_2}^{c-4}$.  Therefore:
\[\pushQED\qed 0\leq \lim_{c\to \infty}{\varepsilon_2(c)}\leq \lim_{c\to \infty}6\left(\frac{{x_2}}{2}\right)^{c-4}=0.
\qedhere\]
\end{proof}

The second part of Theorem~\ref{T:ccave} follows immediately from Theorem~\ref{T:Ep2}.

\bibliographystyle{plain}
\bibliography{myReferences}

\end{document}